\documentclass[11pt]{elsarticle}

\usepackage{amssymb}
\usepackage{amsmath}
\usepackage{mathrsfs}

\usepackage[
top=1.8cm,
bottom=2cm,
left=2cm,
right=2cm
]{geometry}

\usepackage{hyperref}
\usepackage{dsfont}

\newtheorem{theorem}{Theorem}[section]
\newtheorem{lem}[theorem]{Lemma}
\newtheorem{cor}[theorem]{Corollary}
\newtheorem{prop}[theorem]{Proposition}
\newtheorem{defn}[theorem]{Definition}
\newtheorem{rmk}{Remark}

\newtheorem{exa}{Example}
\newenvironment{proof}{\paragraph{Proof:}}{\hfill$\square$}
\numberwithin{equation}{section}

\journal{********}

\begin{document}

\begin{frontmatter}



\title{Averaged Controllability of Time-Fractional Schrödinger Equations with Random Quantum Diffusivity} %% Article title


\author{Jon Asier Bárcena-Petisco$^a$}
\ead{jonasier.barcena@ehu.eus}
\author{Salah-Eddine Chorfi$^b$} %% Author name
\ead{s.chorfi@uca.ac.ma}
\author{Fouad Et-tahri$^{*,c}$}
\ead{fouad.et-tahri@edu.uiz.ac.ma}
\author{Lahcen Maniar$^{b,d}$}
\ead{maniar@uca.ac.ma}

\affiliation{organization={Department of Mathematics, University of the Basque Country UPV/EHU},%Department and Organization
	addressline={Barrio Sarriena s/n}, 
	city={Leioa},
	postcode={48940},
	country={Spain}}
\affiliation{organization={Cadi Ayyad University, UCA, Faculty of Sciences Semlalia, Laboratory of Mathematics, Modeling and Automatic Systems},%Department and Organization
            addressline={B.P. 2390}, 
            city={Marrakesh},
            postcode={40000},
            country={Morocco}}
\affiliation{organization={$^*$ Corresponding author. Faculty of Sciences-Agadir, Lab-SIV, Ibn Zohr University},%Department and Organization
            addressline={B.P. 8106}, 
            city={Agadir},
            postcode={80000},
            country={Morocco}}
\affiliation{organization={The UM6P Vanguard Center, Mohammed VI Polytechnic University},%Department and Organization
            addressline={Hay Moulay Rachid}, 
            city={Ben Guerir},
            postcode={43150},
            country={Morocco}}

%% Abstract
\begin{abstract}
This paper addresses the problem of averaged controllability for the time-fractional Schrödinger equation, where the quantum diffusivity parameter is a random variable with a general probability distribution. First, using the analyticity of the Mittag–Leffler function and Müntz’s theorem, we show that the simultaneous null controllability of the system can hold only for a countable set of realizations of the random diffusivity. In particular, this implies the lack of simultaneous null controllability for absolutely continuous random diffusivity. Then, we establish the lack of exact averaged controllability for absolutely continuous random variables, regardless of the control time. Furthermore, we introduce a new two-parameter fractional characteristic function, allowing us to design a class of random variables satisfying the null averaged controllability at any time from any arbitrary sensor set of positive Lebesgue measure. This can be done using an open-loop control that belongs to $L^\infty$ and is independent of the random parameter. In particular, we obtain the null controllability of the fractional biharmonic diffusion equation. Finally, we conclude with some comments and open problems that deserve future investigation.
\end{abstract}

%\item Addresses averaged controllability for time-fractional Schrödinger equations.
%
%\item Considers quantum diffusivity as a random variable with general distribution.
%
%\item Shows lack of simultaneous controllability for general quantum diffusivity.
%
%\item Introduces a new two-parameter fractional characteristic function for control design.
%
%\item Shows lack of exact averaged controllability for random variables whose fractional characteristic function vanishes at infinity, in particular for absolutely continuous variables.
%
%
%\item Achieves null averaged controllability using open-loop $L^\infty$ controls for suitable random variables.

%% Keywords

\end{frontmatter}



\section{Introduction}
The fractional generalizations of the Schrödinger equation, incorporating time-fractional derivatives, have recently garnered significant interest due to their ability to capture anomalous propagation and memory effects in quantum systems. In \cite{naber2004}, Naber initially derived time-fractional Schrödinger equations by transforming the time-fractional diffusion equation. Alternatively, Achar et al. in \cite{achar2013} obtained the system \eqref{sys:frac Schr} (with constant diffusivity) through the Feynman path integral approach. These fractional models have found applications across a wide spectrum of fields, including quantum mechanics, optics and photonics, condensed matter, and plasma physics; see \cite{laskin2018, caicedo2024} and the references therein. As for inverse problems associated with time-fractional Schrödinger equations, we refer to the recent works \cite{chorfi2025forward} and \cite{chorfi2025Inverse}.

In this paper, we investigate the averaged controllability properties of the time-fractional Schrödinger equation, which is governed by the following equation:
\begin{equation}\label{sys:frac Schr}  (\mathcal{P}_{\xi}) \quad
	\begin{cases}
		\partial^{\alpha}_{0,t} y -\xi \mathrm{i} \Delta y=\mathds{1}_{G_0} u,& \mbox{ on }Q_T,\\
		y=0, & \mbox{ on }\Sigma_T,\\
		y(0,\cdot)=y_0, & \mbox{ in } G.
	\end{cases}
\end{equation}
Here $\partial^{\alpha}_{0,t} y$ denotes the Caputo fractional derivative of order $\alpha\in (0,1)$, which is defined for a suitable function $y$ by
\begin{equation}\label{cap01}
	\partial^{\alpha}_{0,t} y(t)=\frac{1}{\Gamma(1-\alpha)} \int_0^t(t-\tau)^{-\alpha}\partial_\tau y(\tau)\, \mathrm{d} \tau,
\end{equation}
where $\Gamma$ denotes the standard Euler Gamma function.

Above, $T>0$ is a fixed time horizon, $G \subset \mathbb{R}^{d}$ ($d\geq 1$) is a Lipschitz domain with boundary $\partial G$, $Q_T := (0,T)\times G$, $\Sigma_T := (0,T)\times \partial G$, $G_0\subset G$ is a control subset of positive Lebesgue measure, $\mathds{1}_{G_0}$ its indicator function of $G_0$, $\mathrm{i}$ stands for the imaginary unit satisfying $\mathrm{i}^2=-1$, and 
$\xi =\xi (\omega)$ is a real random variable which is defined on a probability space $(\Omega,\mathcal{F},\mathbb{P})$ and naturally induces a probability measure $\mu_{\xi }$ in $\mathbb{R}$, $u=u(t,x)\in L^{\infty}((0,T)\times G)$ is the control, and $y_0\in L^{2}(G)$ is an initial datum which are independent of the random parameter $\omega$. Moreover, we denote by $\{\lambda_n\}_{n\in\mathbb{N}}$ the eigenvalues of the Dirichlet Laplacian $-\Delta$, and by $\{e_n\}_{n\in\mathbb{N}}$ a corresponding Hilbert basis of eigenfunctions in $L^2(G)$. Note that $\lambda_0>0$, and the sequence $\{\lambda_n\}_{n\in\mathbb{N}}$ is non-decreasing and tends to infinity. 

For $\mathbb{P}$-almost every $\omega\in\Omega$ and all $u$ and $y_0$, the solution $y=y(t,x;\xi (\omega); y_0; u)$ of $(\mathcal{P}_{\xi (\omega)})$ is given by the fractional generalization of Duhamel's Principle:
\[y(t,\cdot;\xi (\omega); y_0; u)=y(t,\cdot;\xi (\omega); y_0; 0)+ y(t,\cdot;\xi (\omega);0; u),\]
where
\begin{align*}
	y(t,\cdot;\xi (\omega); y_0; 0)=&\sum_{n=0}^{\infty}\left\langle y_0, e_n\right\rangle E_{\alpha, 1}\left(-\mathrm{i}\xi(\omega)\lambda_n t^\alpha\right) e_n,\\
	y(t,\cdot;\xi (\omega);0; u)=& \sum_{n=0}^{\infty}\int_{0}^{t}\left\langle \mathds{1}_{G_0}u(s,\cdot), e_n\right\rangle (t-s)^{\alpha-1}E_{\alpha,\alpha}(-\mathrm{i}\xi(\omega)\lambda_n(t-s)^{\alpha}))\,\mathrm{d} s\, e_n
\end{align*}
and the two-parameter Mittag--Leffler function is defined by
\begin{equation}\label{MLdef}
	E_{\alpha,\beta}(z)=\sum_{n=0}^{\infty}\frac{z^n}{\Gamma(\alpha n+\beta)} \qquad \forall z\in\mathbb{C},
\end{equation}
which is an entire function generalizing the exponential function; see, for instance, \cite{gorenflo2020} and \cite{naber2004}.

The ideal scenario is to achieve simultaneous controllability for $\mathbb{P}$-almost every $\omega \in \Omega$. Nevertheless, as established in our first main result (Theorem \ref{thm:nonsimcon}), this is impossible when the random variable is absolutely continuous. More precisely, we prove that the preimage of $0$ under the mapping $\xi \mapsto y(T,\cdot;\xi;y_0;u)$ is countable for every $y_0 \in L^2(G)\setminus\{0\}$ and $u \in L^\infty((0,T)\times G_0)$. Consequently, the event “the system \eqref{sys:frac Schr} is simultaneously null controllable” is negligible for absolutely continuous random variables, that is,
$$\mathbb{P}\left[\omega\;:\; y(T,\cdot;\xi (\omega);y_0;u)=0\right]=0.$$

To elucidate how averaging can influence the dynamics of the original system, we examine the uncontrolled version of \eqref{sys:frac Schr} (i.e., $u=0$), where $\xi$ follows the \textbf{Rademacher distribution} defined by
\[
\mathbb{P}(\xi = -1) = \mathbb{P}(\xi = 1) = \frac{1}{2}.
\]
In this case, the mathematical expectation (i.e., the averaged state) of the system
\begin{align*}
	\mathbb{E}(y(t,\cdot;\xi; y_0; 0))=&\frac{1}{2}(y(t,\cdot;-1; y_0; 0)+y(t,\cdot;1; y_0; 0))\\
	=& \frac{1}{2}\sum_{n=0}^{\infty}\left\langle y_0, e_n\right\rangle (E_{\alpha, 1}\left(-\mathrm{i}\lambda_n t^\alpha\right)+E_{\alpha, 1}\left(\mathrm{i}\lambda_n t^\alpha\right)) e_n\\
	=& \sum_{n=0}^{\infty}\left\langle y_0, e_n\right\rangle E_{2\alpha, 1}\left(-\lambda_n^2 t^{2\alpha}\right) e_n,
\end{align*}
which follows from the identity below \eqref{MLi}. Then, the averaged state solves the following fractional biharmonic diffusion-wave equation
\begin{equation*} 
	\begin{cases}
		\partial^{2\alpha}_{0,t} \tilde{y} + \Delta^2 \tilde{y}=0,& \mbox{ on }Q_T,\\
		\Delta\tilde{y}=\tilde{y}=0, & \mbox{ on }\Sigma_T,\\
		\tilde{y}(0,\cdot)=y_0, & \mbox{ in } G,\\
		\tilde{y}_{t}(0,\cdot)=0,\; (\text{if}\; 1<2\alpha<2) & \mbox{ in } G.
	\end{cases}
\end{equation*}
Thus, the fact that averages of Schrödinger-like equations can exhibit improved regularity highlights the importance of this topic.

In the example above, the regularizing effect of diffusion already prevents exact averaged controllability, even if the diffusivity is governed by a discrete random variable. This example led us to introduce a new fractional characteristic function for general random variables (see Definition \ref{df:FCF}). 
Within this framework, we further show that exact averaged controllability also fails for random diffusivities whose fractional characteristic function vanishes at infinity (see Theorem \ref{tm:noexcon}). In particular, this holds for absolutely continuous random diffusivities, relying on a fractional variant of the Riemann--Lebesgue lemma (see Corollary \ref{cor:noexc}).

This raises the question of null averaged controllability, which is therefore of significant interest. In this context, the fractional characteristic function allows us to identify a class of random variables for which null averaged controllability holds, as proved in the main Theorem \ref{thm:nulaver}. The main difficulty is to prove null controllability of the averaged state without explicit knowledge of its dynamics. The identified class assumes exponential decay of the fractional characteristic function and includes for instance the Rademacher distribution. As a result, the fractional biharmonic diffusion system described above is null controllable. 

For the proof, we develop a strategy based on the Fourier series decomposition of the averaged state together with the exponential decay of the fractional characteristic function and a suitable spectral inequality. This allows us to establish an interpolation estimate, and then, using a telescoping argument, we deduce an $L^1$-observability inequality of a suitable adjoint system, which we first show to be equivalent to the null averaged controllability of the original system. 

We would like to remark that in this paper we focus on controls belonging to $L^\infty$, as they better reflect the underlying physical interpretation compared to controls in $L^2$. However, note that a control in $L^\infty$ belongs to $L^2$, so this method would also work if we look for controls in $L^2$, as it is usually done.

The structure of the rest of the paper is as follows. In the rest of the section, we present the state of the art and the notation.
In Section \ref{Sec2}, we introduce the key concepts on fractional calculus and average controllability as well as their primary characterizations. In Sections \ref{Sec3} and \ref{Sec4}, we establish the lack of simultaneous controllability as well as exact averaged controllability for general absolutely continuous random variables. Section \ref{Sec5} will be devoted to establishing the null averaged controllability of system \eqref{sys:frac Schr} for a certain class of random variables. Section \ref{Sec6} presents some concluding remarks and open questions that deserve further investigation.

\subsection*{State of the art on averaged controllability}
The study of averaged controllability and simultaneous controllability with abstract formulations originates in \cite{zuazua2014averaged} and \cite{loheac2016} for finite-dimensional systems. An extension of these concepts to PDEs appeared in \cite{lu2016averaged}, where the authors investigated the control of averaged states for the transport, heat, and Schrödinger equations under various probabilistic frameworks. In particular, for the Schrödinger equation, they considered diffusivities following uniform, exponential, normal, Laplace, chi-square, and Cauchy distributions. The case of the continuously averaged heat equation was subsequently analyzed in \cite{zuazua2016stable}.  Further contributions, including \cite{lazar2018stability} and \cite{loheac2017averaged}, addressed perturbations of the probability density function modeled by Dirac measures. The works \cite{barcena2021averaged} and \cite{coulson2019average} investigated the controllability of the heat equation with random diffusivity, demonstrating that certain probability laws give rise to fractional dynamics. In \cite{abdelli2021}, the authors proposed a numerical method to approximate the exact averaged boundary control of a family of wave equations by projecting the control problem in the finite-dimensional space formed by the first eigenfunctions.

In addition, numerous results concerning lower-order random perturbations have been reported; see the monograph \cite{lu2021mathematical} and the survey \cite{lu2022concise}, along with more recent developments on nonlinear parabolic stochastic equations, such as those in \cite{hernandez2022statistical} and \cite{hernandez2023global}.

Recently, the authors of \cite{boutaayamou2025} have investigated the controllability of finite-dimensional systems governed by fractional-order dynamics while considering uncertain parameters. They have also characterized the averaged controllability by means of an average Kalman rank condition as well as an average Gramian matrix. Moreover, the paper \cite{barcena2025} has proved that the average solutions to a random Schrödinger equation with absolutely continuous diffusivity can be driven to zero using deterministic controls. The authors have also proven the lack of average exact controllability, as well as the failure of simultaneous null controllability, except for finite realizations of the random diffusivity. In the present paper, we mainly aim to extend existing results to the time-fractional setting. Such extensions require new techniques and ideas involving the Mittag--Leffler functions with the imaginary argument, which present several challenges. We refer to \cite{xue2017observability} for an observability inequality associated with fractional diffusion equations with constant diffusion (the analogue of the heat equation), and to \cite{fujishiro2014} for approximate controllability for fractional diffusion equations by interior control. Finally, we should emphasize that the literature is scarce regarding observability and controllability for time-fractional PDEs, specially if we compare to amount of results with fractional derivative in the space variable (see, for example,  \cite{biccari2019controllability,biccari2025boundary,koenig2020lack,micu2006controllability}).

\subsection*{Notational setting}
Here, we introduce standard notation and conventional abbreviations from probability theory for future use.

We denote by $\mathbb{N}=\{0,1,2, \cdots\}$ the set of natural integers. We denote by $|\mathscr{A}|$ the Lebesgue measure of a measurable set $\mathscr{A}\subset \mathbb{R}^d$ ($d\geq 1$). The notation $B(x,r)$ stands for the open ball centered at $x\in\mathbb{R}^d$ with radius $r>0$. We consider the Hilbert space $L^2(G) := L^{2}(G;\mathbb{C})$ endowed with its usual Hermitian inner product $\left\langle\cdot,\cdot\right\rangle_{L^2(G)}$, or $\left\langle\cdot,\cdot\right\rangle$ if there is no ambiguity, and $\|\cdot\|_{L^2(G)}$ denotes its associated norm. The space of entire functions will be denoted by \(\mathcal{H}(\mathbb{C})\). For $T>0$ and a given Hilbert space $(H,\langle\cdot,\cdot\rangle_H)$, the spaces $L^p(0,T;H)$ $(1\le p\le\infty)$ and $C([0,T];H)$ denote the standard function spaces, whereas $W^{1,1}(0,T;H)$ designates the space of absolutely continuous $H$-valued functions on $[0,T]$.

Next, we consider probability space $(\Omega,\mathcal{F},\mathbb{P})$. An $L^2(G)$-valued Borel random variable is a function from $\Omega$ to $L^2(G)$ that is $(\mathcal{F},\mathcal{B}(L^2(G)))$-measurable, where $\mathcal{B}(L^2(G))$ denotes the Borel $\sigma$-algebra on $L^2(G)$. A real Borel random variable will shortly be called a real random variable. We consider a real random variable $\xi\colon \Omega \to \mathbb{R}$ representing the quantum diffusivity. The probability distribution of $\xi$, defined by $\mu_{\xi } := \mathbb{P}\circ \xi ^{-1}$, will be called its distribution. The Probability Density Function of an absolutely continuous random variable will shortly be denoted by \textbf{PDF}.

We also recall that the distribution of $\xi$ is absolutely continuous (with respect to the Lebesgue measure) if there is a nonnegative density $\rho_{\xi }\in L^1(\mathbb R)$ such that for all Borel measurable set $\mathscr{A}\subset\mathbb R$,
$$\mathbb \mu_{\xi }[\mathscr{A}]:= \mathbb{P}[\xi \in \mathscr{A}]=\int_\mathscr{A} \rho_{\xi }(x)\mathrm{d}x.$$
In this case, we simply say that $\xi$ is absolutely continuous.

For a continuous function $F:\mathbb{R}\to L^2(G)$, $F\circ\xi $ is an $L^2(G)$-valued Borel random variable whose expectation (whenever it exists) is given by the formula:
\begin{equation*}
	\mathbb{E}(F\circ\xi ):=\int_{\Omega}F(\xi (\omega))\,\mathrm{d}\mathbb{P}(\omega) =\int_{-\infty}^{\infty}F(\xi)\,\mathrm{d}\mu_{\xi }(\xi).
\end{equation*}
Finally, the notation $\sim$ denotes asymptotic equivalence, while the symbol $\mathcal{O}$ indicates that the ratio of two quantities remains bounded in the considered asymptotic limit.
\section{Miscellanea on fractional calculus and averaged control} \label{Sec2}
\subsection{Fractional calculus} \label{FDI}
We set
$$\omega_{\alpha}(t)=\begin{cases}
	\dfrac{1}{\Gamma(\alpha)t^{1-\alpha}} \quad &t>0,\\
	0 & t\leq 0,
\end{cases}$$
where $\alpha>0$ and $\Gamma$ is the standard Euler function.

The left and right Riemann–Liouville fractional integrals of order $\alpha>0$ are respectively given 
by the convolution with $\omega_\alpha \in L^1(0,T)$, namely:
\begin{align*}
	(I^{\alpha}_{0,t}g)(t)&:=(\omega_{\alpha}\ast g)(t) =\int_{0}^{t}\omega_{\alpha}(t-s)g(s)\mathrm{d}s, \\
	(I^{\alpha}_{t,T}g)(t)&:=(\omega_{\alpha}\ast g(T-\cdot))(T-t) =\int_{t}^{T}\omega_{\alpha}(s-t)g(s)\mathrm{d}s,
\end{align*}
whenever the right-hand side is well-defined. Invoking the properties of the convolution product, we deduce that \cite{jin2021fractional}:
\begin{lem}\label{Continuity of fractional integrals}
	Let $\alpha>0$ and $1\le p\le \infty$. Then $I^{\alpha}_{0,t}$ acts continuously on $L^{p}(0,T;H)$, and we have
	\begin{equation*}
		\|I^{\alpha}_{0,t}g\|_{L^{p}(0,T;H)}\leq \frac{T^{\alpha}}{\Gamma(\alpha+1)}\|g\|_{L^{p}(0,T;H)}.
	\end{equation*}
\end{lem}
The same holds for the right integral $I^{\alpha}_{t,T}$.

In the following, we consider a dense subspace $V \subset H$ continuously 
embedded into $H$. We denote by $V^{\prime}$ the dual of $V$ with respect to the pivot space $H$, and by $\langle\cdot,\cdot\rangle_{V^\prime, V}$ the duality pairing 
between $V^\prime$ and $V$. We recall that
\[ \langle f,g \rangle_{V^\prime, V}=\langle f,g\rangle_H,\quad f\in H,\; g\in V. \]
We will use the following fractional integration by parts formulas; see \cite{jin2021fractional}.
\begin{lem}\label{Fractional integration by parts}
	Let $f\in L^p(0,T;V^{\prime})$ and $g\in L^q(0,T;V)$ with $p,q\ge 1$ and $\alpha>0$ satisfying $\frac{1}{p}+\frac{1}{q}\le 1+\alpha$. 
	Then, the following identity holds:
	\begin{equation*}
		\int_{0}^{T} \langle (I^{\alpha}_{0,t}f)(t), g(t)\rangle_{V^{\prime},V} \mathrm{d} t=\int_{0}^{T}\langle f(t), (I^{\alpha}_{t,T}g)(t)\rangle_{V^{\prime},V} \mathrm{d} t.
	\end{equation*}
\end{lem}

Next, we introduce the definitions of the Caputo and Riemann–Liouville fractional derivatives.
\begin{itemize}
	\item  The left and right Caputo fractional derivatives of order $\alpha\in (0,1)$ are respectively defined by
	\begin{equation}
		\partial^{\alpha}_{0,t}g(t)=(I^{1-\alpha}_{0,t}g^{\prime})(t) \quad\mbox{and}\quad \partial^{\alpha}_{t,T}g(t)=(I^{1-\alpha}_{t,T}g^{\prime})(t)
	\end{equation}
	when the right-hand sides are well-defined. Note that, if $g\in W^{1,1}(0,T;H)$, then $\partial^{\alpha}_{0,t}g, \partial^{\alpha}_{t,T}g$ exist and $\partial^{\alpha}_{0,t}g, \partial^{\alpha}_{t,T}g\in L^1(0,T;H)$.
	\item The left and right Riemann–Liouville fractional derivatives of order $\alpha\in (0,1)$ are respectively defined by
	\begin{equation*}
		D^{\alpha}_{0,t}g(t)=\frac{\mathrm{d}}{\mathrm{d} t}(I^{1-\alpha}_{0,t}g)(t) \quad\mbox{and}\quad D^{\alpha}_{t,T}g(t)=-\frac{\mathrm{d}}{\mathrm{d} t}(I^{1-\alpha}_{t,T}g)(t)
	\end{equation*}
	whenever the right-hand sides are well-defined. Note that, if $g\in W^{1,1}(0,T;H)$, then $D^{\alpha}_{0,t}g, D^{\alpha}_{t,T}g$ exist and belong to $L^1(0,T;H)$.
\end{itemize}

The next integration formula allows us to derive the adjoint system associated with the forward system involving the Riemann–Liouville fractional derivative.\\
\begin{prop}
	Let $\alpha\in (0,1)$ and $T>0$. We assume that
	\begin{itemize}
		\item $f\in C([0,T];H)$ such that $\partial^{\alpha}_{0,t}f \in L^{\infty}(0,T;V^{\prime})$ and $f^{\prime}\in L^{q}(0,T;V^{\prime})$ for some  $q\geq\frac{1}{1-\alpha}$,
		\item $g\in L^{1}(0,T;V)$ such that $D^{\alpha}_{t,T}g\in L^1(0,T;V)$.
	\end{itemize}
	Then 
	\begin{equation}\label{IBPF}
		\int_{0}^{T} \langle\partial^{\alpha}_{0,t}f(t), g(t)\rangle_{V^{\prime},V} \mathrm{d} t=\left[\langle f(t), (I^{1-\alpha}_{t,T}g)(t) \rangle_H\right]_{t=0}^{t=T} +\int_{0}^{T} \langle f(t), D^{\alpha}_{t,T}g(t)\rangle_{H} \mathrm{d} t.
	\end{equation}
\end{prop}
\begin{proof}
	Using the Caputo derivative definition, we rewrite
	\begin{align*}
		\int_{0}^{T} \langle\partial^{\alpha}_{0,t}f(t), g(t)\rangle_{V^{\prime},V}\mathrm{d} t=\int_{0}^{T} \left\langle (I^{1-\alpha}_{0,t}f^{\prime})(t), g(t)\right\rangle_{V^{\prime},V} \mathrm{d} t.
	\end{align*}
	By Lemma \ref{Fractional integration by parts} for $1-\alpha$ (instead of $\alpha$), $p=1$ and $q\geq \frac{1}{1-\alpha}$, we obtain
	\begin{align}
		\int_{0}^{T} \langle\partial^{\alpha}_{0,t}f(t), g(t)\rangle_{V^{\prime},V}\mathrm{d} t=\int_{0}^{T} \left\langle f^{\prime}(t), (I^{1-\alpha}_{t,T}g)(t) \right\rangle_{V^{\prime},V} \mathrm{d} t. \label{Integral}
	\end{align}
	Applying the standard integration by parts in $W^{1,1}(0,T;V^{\prime})\times W^{1,1}(0,T;V)$ to $f$ and $I^{1-\alpha}_{t,T}g$, the integral on the right-hand side of \eqref{Integral} yields \eqref{IBPF}.
\end{proof}
\begin{rmk} 
	$(I^{1-\alpha}_{t,T}g)(t)$ is well defined as an element of $V$ for all $t\in [0,T]$, due to $I^{1-\alpha}_{t,T} g \in W^{1,1}(0,T;V)$, since $g \in L^1(0,T;V)$ (then $I^{1-\alpha}_{t,T} g\in L^1(0,T;V)$ see Lemma \ref{Continuity of fractional integrals}) and $\frac{\mathrm{d}}{\mathrm{d} t}(I^{1-\alpha}_{t,T}g)(t)=-D^{\alpha}_{t,T}g \in L^1(0,T;V)$.
\end{rmk}
\begin{rmk}
	Note that, by a density argument, \eqref{IBPF} remains valid even without assuming any regularity of the usual derivative ($f\in C([0,T];H)$ such that $\partial^{\alpha}_{0,t}f \in L^{\infty}(0,T;V^{\prime})$ and $g\in L^{1}(0,T;V)$ such that $D^{\alpha}_{t,T}g\in L^1(0,T;V)$).
\end{rmk}

Recall that, for $\alpha>0$ and $\beta\in \mathbb{R}$, the Mittag--Leffler function $E_{\alpha,\beta}(z)$ defined in \eqref{MLdef} is an entire function. Moreover, we have the following lemmas; we refer to \cite[Pages 32-35]{podlubny99} for the proofs.

The first lemma concerns the asymptotic expansion of the Mittag--Leffler functions as $|z|\to \infty$.
\begin{lem}
	Let $\alpha\in(0,2)$ and $\beta \in \mathbb{R}$ be arbitrary. Let $\mu$ be such that $\frac{\alpha\pi}{2}<\mu< \min \{\pi, \pi \alpha\}$. Then, for any $p\ge 1$,
	\begin{equation}\label{asymptotic expansions}
		E_{\alpha, \beta}(z)=-\sum_{k=1}^p\frac{1}{\Gamma(\beta-\alpha k)}\frac{1}{z^k} + \mathcal{O}\left(\frac{1}{|z|^{1+p}} \right), \; |z|\to \infty, \; \mu \leq |\arg (z)| \leq \pi.
	\end{equation}
\end{lem}

Consequently, we also have the following boundedness result.
\begin{lem}\label{MLbd}
	Let $\alpha\in(0,1)$ and $\beta\in \mathbb{R}$. There exists a constant $C_0>0$ depending only on $\alpha$ and $\beta$ such that
	\begin{align}\label{es0}
		& \left|E_{\alpha, \beta}(\mathrm{i} t)\right| \leq \frac{C_0}{1+|t|} \leq C_0\qquad \forall t\in \mathbb{R}.
	\end{align}
\end{lem}

We propose a new definition for the fractional characteristic function, which is of independent interest:
\begin{defn} \label{df:FCF}
	Let $\xi$ be a real random variable. We define the \textbf{two-parameter Fractional Characteristic Function (FCF)} of $\xi$ by
	\begin{equation}\label{def:fractional charfunc}
		\varphi_{\alpha,\beta}^{(\xi)} (s) := \mathbb{E}(E_{\alpha,\beta}(\mathrm{i}s\xi))=\int_{-\infty}^{\infty}E_{\alpha,\beta}(\mathrm{i}s\xi)\,\mathrm{d}\mu_{\xi }(\xi), \qquad s\in \mathbb{R}.
	\end{equation}
	In the case $\alpha = \beta = 1$, we recover the usual characteristic function. Note that this definition is different from its counterpart in \cite{tomovski2022}.
\end{defn}

The following lemma generalizes certain properties of classical characteristic functions, such as boundedness, as well as the Riemann–Lebesgue lemma for absolutely continuous random variables.
\begin{lem}\label{lm:van der Corput}
	Let $\alpha\in (0,1)$ and $\beta>0$. For any random variable $\xi$ on $(\Omega, \mathcal{F}, \mathbb{P})$, we have
	\begin{enumerate}
		\item $\varphi_{\alpha, \beta}^{(\xi)}(s)$ is well-defined for all real $s$.
		\item $|\varphi_{\alpha, \beta}^{(\xi)}(s)|\le C_0$ for all $s\in\mathbb{R}$,
		where $C_0$ is the constant introduced in Lemma \ref{MLbd}.
		\item Moreover, if $\xi$ is an absolutely continuous random variable, then 
		$$\varphi_{\alpha, \beta}^{(\xi)}(s)\to 0\quad \mbox{as}\;\; |s|\to \infty.$$
	\end{enumerate}
\end{lem}
\begin{proof}
	Using Lemma \ref{MLbd}, we obtain
	\[|E_{\alpha,\beta}(\mathrm{i}s\xi)|\le C_0,\quad s\in\mathbb{R}.\]
	Then, the first point follows from the fact that the probability measure is finite, and we have
	\begin{align*}
		|\varphi_{\alpha, \beta}^{(\xi)}(s)|\le& \mathbb{E}(|E_{\alpha,\beta}(\mathrm{i}s\xi)|)\\
		\le& \mathbb{E}(C_0)=C_0,\quad s\in\mathbb{R}.
	\end{align*}
	Moreover, if $\xi$ is an absolutely continuous random variable, then its \textbf{PDF} satisfies $\rho_{\xi}\in L^{1}(\mathbb{R})$, and its \textbf{FCF} is given by
	\[\varphi_{\alpha, \beta}^{(\xi)}(s)=\int_{-\infty}^{\infty}E_{\alpha,\beta}(\mathrm{i}s\xi)\rho_{\xi}(\xi)\,\mathrm{d}\xi,\quad s\in\mathbb{R}.\]
	Using Lemma \ref{MLbd}, we obtain for almost every $\xi\in\mathbb{R}$
	\begin{align*}
		|E_{\alpha,\beta}(\mathrm{i}s\xi)\rho_{\xi}(\xi)|\le& \frac{C_0}{1+|s||\xi|}\rho_{\xi}(\xi)\to 0\quad\mbox{as}\; |s|\to\infty, \nonumber\\
		\le& C_0 \rho_{\xi}(\xi)\in L^1(\mathbb{R}).
	\end{align*}
	Then, the third point follows directly from the dominated convergence theorem.
\end{proof}
\begin{rmk}
	In the situation $0<\alpha\le 1$ and $\beta\ge \alpha$, we have
	\begin{align*}
		|\varphi_{\alpha, \beta}^{(\xi)}(s)|\le& \frac{1}{\Gamma(\beta)},\quad s\in\mathbb{R},
	\end{align*}
	due to $|E_{\alpha,\beta}(z)|\le E_{\alpha,\beta}(\mathrm{Re}(z))$ for all $z\in\mathbb{C}$ (see Theorem 1 in \cite{garrappa2025some}). In particular, when $\alpha=\beta=1$, we recover the fact that the usual characteristic is bounded by $1$.
\end{rmk}
\begin{exa} \label{Ex: Rademacher distribution}
	An interesting example of a random variable is given by the Rademacher distribution, defined
	by
	\[
	\mathbb{P}(\xi = -1) = \mathbb{P}(\xi = 1) = \frac{1}{2},
	\]
	which is widely used in applications requiring symmetric discrete randomness, such as randomized algorithms and statistical simulations.
\end{exa}
\begin{prop}\label{prop:FCFRade}
	The Rademacher distribution has an \textbf{FCF} given by
	\[
	\varphi_{\alpha,\beta}^{(\xi)}(s) = E_{2\alpha,\beta}(-s^2),\quad s\in\mathbb{R}.
	\]
\end{prop}
\begin{proof}
	For the Rademacher variable $\xi$, we have
	\begin{align*}
		\varphi_{\alpha,\beta}^{(\xi)}(s)&= E_{\alpha,\beta}(-\mathrm{i}s)\mathbb{P}(\xi=-1) + E_{\alpha,\beta}(\mathrm{i}s)\mathbb{P}(\xi=1)\\
		&= \frac{1}{2}\left(E_{\alpha,\beta}(-\mathrm{i}s)+E_{\alpha,\beta}(\mathrm{i}s)\right)\\
		&= \mathrm{Re}~E_{\alpha,\beta}(\mathrm{i}s)\\
		&=  E_{2\alpha,\beta}(-s^2).
	\end{align*}
	Note that we have used the identity
	\begin{equation}\label{MLi}
		E_{\alpha,\beta}(\mathrm{i} s)=E_{2 \alpha, \beta}\left(-s^2\right)+\mathrm{i} s E_{2 \alpha, \alpha+\beta}\left(-s^2\right), \quad s\in \mathbb{R}.
	\end{equation}
\end{proof}
%\textcolor{red}{
	%\begin{remark}
	%This result is surprising and unexpected, since the usual characteristic function (i.e., $\alpha=\beta=1$) of a discrete random variable never tends to zero at infinity.
	%\end{remark}
	%}
\subsection{Averaged controllability}
First, we provide some well-posedness results for the forward and adjoint problems under consideration.
\begin{prop} \label{FP}
	Let $\alpha\in(0,1)$, $T>0$, $\xi\in \mathbb{R}^*$ a nonzero real,  $y_0\in L^{2}(G)$ and $u\in L^{\infty}(0,T;L^{2}(G))$. Then, there exists a unique weak solution $y \in C([0,T];L^{2}(G))$ to the system \eqref{sys:frac Schr} satisfying
	\begin{align}
		\|y\|_{C([0,T];L^{2}(G))}&\le C \left(\|y_0\|_{L^2(G)}+\|u\|_{L^\infty(0,T;L^2(G))}\right),\label{stabest1}\\
		\|\partial^{\alpha}_{0,t} y\|_{L^{\infty}(0,T;H^{-1}(G))} &\leq C \left(\|y_0\|_{L^2(G)}+\|u\|_{L^\infty(0,T;L^2(G))}\right), \label{stabest2}
	\end{align}
	for some constant $C>0$. Moreover, the unique solution is given by the formula
	\begin{equation}\label{solution formula}
		\begin{split}
			y(t,\cdot)=\sum_{n=0}^{\infty}\left[\left\langle y_0, e_n\right\rangle E_{\alpha, 1}\left(-\mathrm{i}\xi\lambda_n t^\alpha\right) + \int_{0}^{t}\left\langle \mathds{1}_{G_0}u(s,\cdot), e_n\right\rangle (t-s)^{\alpha-1}E_{\alpha, \alpha}\left(-\mathrm{i}\xi\lambda_n (t-s)^\alpha\right)\,\mathrm{d} s\right] e_n.
		\end{split}
	\end{equation}
\end{prop}

\begin{rmk}
	The solution is essentially a Duhamel formula. Indeed, as proved in \cite{naber2004}:
	\[\partial^\alpha_t\Big(E_{\alpha,\beta}(t^\alpha)\Big)=E_{\alpha,\beta}(t^\alpha).\]
\end{rmk}

\begin{proof}
	Let $y_n \in C\left(\mathbb{R}^{+};\mathbb{C}\right)$ be defined by
	\begin{equation}\label{yn}
		y_n(t)= E_{\alpha, 1}\left(-\mathrm{i}\xi\lambda_n t^\alpha\right) y_{0, n} +\int_0^t k_n(t-s) u_n(s) \mathrm{d} s, \quad t>0,
	\end{equation}
	where $y_{0, n}=\left\langle y_0, e_n\right\rangle$, $k_n(s)=s^{\alpha-1}E_{\alpha, \alpha}\left(-\mathrm{i}\xi\lambda_n s^\alpha\right)$ and $u_n(s)=\left\langle \mathds{1}_{G_0} u(s,\cdot), e_n\right\rangle \mathds{1}_{(0, T)}(s)$. 
	
	We next prove that the series $\displaystyle\sum_{n \geq 0} y_n(t) e_n(x)$ converges to a weak solution of \eqref{sys:frac Schr}.
	
	By Lemma \ref{MLbd}, for all $t>0$, and all $p, q \in\mathbb{N}$, we have:
	$$
	\begin{aligned}
		\left\|\sum_{n=p}^{p+q} y_n(t) e_n\right\|_{L^2(G)} &\leq  C_0\left\|\sum_{n=p}^{p+q} y_{0, n} e_n\right\|_{L^2(G)}+\int_0^t \left\|\sum_{n=p}^{p+q} k_{n}(t-s) u_n(s) e_n\right\|_{L^2(G)} \mathrm{d}s\\
		& \le C_0\left\|\sum_{n=p}^{p+q} y_{0, n} e_n\right\|_{L^2(G)}+\int_0^t \sup_{n\ge p}|k_n(t-s)|\left\|\sum_{n=p}^{p+q} u_n(s) e_n\right\|_{L^2(G)}\mathrm{d}s\\
		&\le C_0\left\|\sum_{n=p}^{p+q} y_{0, n} e_n\right\|_{L^2(G)}+\|u\|_{L^\infty(0,T;L^2(G))}\int_0^t \sup_{n\ge p}|k_n(t-s)|\mathrm{d}s\\
		&= C_0\left\|\sum_{n=p}^{p+q} y_{0, n} e_n\right\|_{L^2(G)}+\|u\|_{L^\infty(0,T;L^2(G))}\int_0^t \sup_{n\ge p}|k_n(s)|\mathrm{d}s.
	\end{aligned}
	$$
	Then, for all $\tau>0$, we obtain:
	\begin{equation}\label{estkn}
		\sup_{t \in [0, \tau]}\left\|\sum_{n=p}^{p+q} y_n(t) e_n\right\|_{L^2(G)} \leq C_0\left\|\sum_{n=p}^{p+q} y_{0, n} e_n\right\|_{L^2(G)}+ \|u\|_{L^\infty(0,T;L^2(G))}\int_0^\tau \sup_{n\ge p}|k_n(s)|\mathrm{d}s.
	\end{equation}
	By Lemma \ref{MLbd}, we have for all $n\ge p$ and $s\in (0,\tau]$,
	\begin{align*}
		|k_n(s)|
		=& s^{\alpha-1}\,\bigl|E_{\alpha,\alpha}(-\mathrm{i}\xi\lambda_n s^\alpha)\bigr|\\
		\le & \frac{C_0}{1+|\xi|\lambda_{p}s^{\alpha}}s^{\alpha-1}\\
		\le& C_0\,s^{\alpha-1}.
	\end{align*}
	Then, 
	\begin{align*}
		&\sup_{n\ge p}|k_n(s)|
		\le C_0\,s^{\alpha-1} \in L^1(0,\tau),\\
		& \sup_{n\ge p}|k_n(s)| \longrightarrow 0 \qquad \text{as } p \to\infty.
	\end{align*}
	Hence, by the dominated convergence theorem,
	\[
	\int_0^T \sup_{n\ge p} |k_n(\tau)|\, d\tau
	\;\xrightarrow[p\to\infty]{}\; 0.
	\]
	Combining this with \eqref{estkn} yields,
	$$
	\lim _{p \rightarrow \infty} \sup_{t \in [0, \tau]}\left\|\sum_{n=p}^{p+q} y_n(t) e_n\right\|_{L^2(G)}=0\qquad \text{for all }\; q\in\mathbb{N}.
	$$
	Thus, for any $\tau>0$ the series $\displaystyle\sum_{n \geq 0} y_n(t) e_n$ converges uniformly in $t \in [0, \tau]$ to $t\mapsto\widetilde{y}(t)=\displaystyle\sum_{n =0}^{\infty} y_n(t) \in C\left([0,\tau], L^2(G)\right)$. Hence $\widetilde{y}\in C\left(\mathbb{R}^+, L^2(G)\right)$.
	Moreover, we obtain \eqref{stabest1} for $\widetilde{y}$.
	
	On the other hand, by the Laplace transform, we obtain
	$$\partial^{\alpha}_{0,t} y_n(t)=-\mathrm{i}\xi\lambda_n y_n(t) + u_n(t).$$
	Then, we infer that
	\begin{equation*}
		|\partial^{\alpha}_{0,t} y_n(t)| \le |\xi| \lambda_n |y_n(t)|+|u_n(t)|,
	\end{equation*}
	which implies that
	\begin{align*}
		\|\partial^{\alpha}_{0,t} \widetilde{y}\|^2_{H^{-1}(G)} &=\sum_{n=0}^\infty \frac{1}{\lambda_n^2}|\partial^{\alpha}_{0,t} y_n(t)|^2 \\
		& \le 2\sum_{n=0}^\infty |\xi|^2|y_{n}(t)|^2 + 2\sum_{n=0}^\infty \frac{1}{\lambda_n^2}\left|u_n(t)\right|^2\\
		&\le C \|y_0\|^2_{L^2(G)} + C\|u\|^2_{L^\infty(0,T;L^2(G))}.
	\end{align*}
	Hence, $y=\widetilde{y}\big\rvert_{(0,T)\times G}$ is a weak solution of \eqref{sys:frac Schr} belonging to $C([0,T];L^2(G))$ and satisfying \eqref{stabest1}-\eqref{stabest2}. The uniqueness of the solution can be obtained via the Laplace transform.
\end{proof}

To formulate the backward system corresponding to \eqref{sys:frac Schr}, we consider a smooth function $z=z(t,x)$.
\begin{align*}
	& \int_{0}^{T}\langle \partial^{\alpha}_{0,t} y(t,\cdot)-\xi \mathrm{i} \Delta y(t,\cdot), z(t,\cdot)\rangle_{H^{-1}(G), H^{1}_{0}(G)} \mathrm{d} t\\
	&= \int_{0}^{T}\langle \partial^{\alpha}_{0,t} y(t,\cdot), z(t,\cdot)\rangle_{H^{-1}(G), H^{1}_{0}(G)} \mathrm{d} t-\int_{0}^{T}\langle \xi \mathrm{i} \Delta y(t,\cdot), z(t,\cdot)\rangle_{H^{-1}(G), H^{1}_{0}(G)} \mathrm{d} t \\
	&= \left[\langle y(t,\cdot), (I^{1-\alpha}_{t,T}z)(t,\cdot) \rangle_{L^2(G)}\right]_{t=0}^{t=T} +\int_{0}^{T} \langle y(t,\cdot), D^{\alpha}_{t,T}z(t,\cdot)\rangle_{H^{-1}(G), H^{1}_{0}(G)} \mathrm{d} t\\
	& +\int_{0}^{T}\langle  y(t,\cdot), \xi \mathrm{i} \Delta z(t,\cdot)\rangle_{H^{-1}(G), H^{1}_{0}(G)} \mathrm{d} t\\
	& = \left[\langle y(t,\cdot), (I^{1-\alpha}_{t,T}z)(t,\cdot) \rangle_{L^2(G)}\right]_{t=0}^{t=T} +\int_{0}^{T} \langle y(t,\cdot), D^{\alpha}_{t,T}z(t,\cdot)+\xi \mathrm{i} \Delta z(t,\cdot)\rangle_{H^{-1}(G), H^{1}_{0}(G)} \mathrm{d} t.
\end{align*}
So, we consider the following adjoint system
\begin{equation}\label{sys:adjSchr}
	\begin{cases}
		D^{\alpha}_{t,T} z +\xi  \mathrm{i} \Delta z=0,& \mbox{ on }Q_T,\\
		z=0, & \mbox{ on }\Sigma_T,\\
		I^{1-\alpha}_{t,T} z|_{t=T}=z_T, & \mbox{ in } G,
	\end{cases}
\end{equation}
where $D^{\alpha}_{t,T}$ (resp. $I^{1-\alpha}_{t,T}$) denotes the right Riemann-Liouville time fractional derivative of order $\alpha$
(resp. the right Riemann-Liouville time fractional integral of order $1-\alpha$) as introduced in Subsection \ref{FDI}.\\
As for the well-posedness of the adjoint system \eqref{sys:adjSchr}, we have the following result:
\begin{prop}
	Let $\alpha\in (0,1)$, $T>0,$ $\xi\in \mathbb{R}^*$ a nonzero real and $z_T\in L^2(G)$. Then, the adjoint system \eqref{sys:adjSchr} admits a unique solution $z$ such that:
	\begin{align}
		z(t,\cdot)=\sum_{n=0}^{\infty} (T-t)^{\alpha-1}E_{\alpha,\alpha}(-\mathrm{i}\xi\lambda_n (T-t)^{\alpha})) \left\langle z_T, e_n\right\rangle e_n, \label{sol:zsum}\\
		I^{1-\alpha}_{t,T} z(t,\cdot)=\sum_{n=0}^{\infty} E_{\alpha,1}(-\mathrm{i}\xi\lambda_n (T-t)^{\alpha})) \left\langle z_T, e_n\right\rangle e_n,\label{sol:zsumrev}
	\end{align}
	and there exists a constant $C>0$ (independent of $z_T$) such that:
	\begin{equation*}
		\|I^{1-\alpha}_{t,T} z\|_{C([0,T];L^2(G))} \le C \|z_T\|_{L^2(G)}.
	\end{equation*}
	Moreover,
	\begin{itemize}
		\item[(i)] $z\in C([0,T); H^2(G)\cap H^1_0(G))$ and $D^{\alpha}_{t,T}z \in C([0,T); L^2(G))$.
		\item[(ii)] If $z_T\in H^1_0(G)$, then $z\in L^1(0,T; H^1_0(G))$, and if $z_T\in D((-\Delta)^{3/2})$, then $D^{\alpha}_{t,T}z \in L^1(0,T; H^1_0(G))$. In addition, there exists a constant $C>0$ (independent of $z_T$) such that
		\begin{align*}
			\|z\|_{L^1(0,T; H^1_0(G))} &\le C \|z_T\|_{H^1_0(G)},\\
			\|D^{\alpha}_{t,T}z\|_{L^1(0,T; H^1_0(G))} &\le C \|z_T\|_{D((-\Delta)^{3/2})}.
		\end{align*}
		%\item[(iii)] $z\colon [0,T) \to L^2(G)$ is analytically extended to the half-plane $\{\eta\in \mathbb{C}\colon \mathrm{Re}~\eta<T\}$.
	\end{itemize}
\end{prop}
The proof is an adaptation of \cite[Proposition 4.1]{fujishiro2014}, so it is omitted.

We now introduce the following notions of averaged controllability. In this paper we focus on controls belonging to $L^\infty$, as they better reflect the underlying physical interpretation compared to controls in $L^2$:
\begin{defn}\label{def-ex-con}
	System \eqref{sys:frac Schr} is exactly averaged controllable in $L^2(G)$ with control cost $\mathscr{C}_{ex}=\mathscr{C}_{ex}(G, G_0, \xi , T)$ (positive constant) if for all $y_0, y_1\in L^2(G)$, there exists a control $u\in L^\infty((0,T)\times G_{0})$ such that
	\begin{equation}\label{def-ex-con-eq1}
		\|u\|_{L^\infty((0,T)\times G_0)}\leq \mathscr{C}_{ex}(\|y_0\|_{L^2(G)} + \|y_1\|_{L^2(G)})
	\end{equation}
	and the average solution to \eqref{sys:frac Schr} satisfies
	\begin{equation*}
		\mathbb{E}(y(T,\cdot;\xi ; y_0; u))=y_1.
	\end{equation*}
\end{defn}

\begin{defn} 
	System \eqref{sys:frac Schr} is null averaged controllable in $L^2(G)$ with control cost \\$\mathscr{C}_{null}=\mathscr{C}_{null}(G,G_0, \xi , T)$ if for all $y_0\in L^2(G)$, there exists a control $u\in L^\infty((0,T)\times G_0)$ such that
	\begin{equation}\label{est:connull}
		\|u\|_{L^\infty((0,T)\times G_0)}\leq \mathscr{C}_{null}\|y_0\|_{L^2(G)}
	\end{equation}
	and the average solution to \eqref{sys:frac Schr} satisfies
	\begin{equation}
		\mathbb{E}(y(T,\cdot;\xi ; y_0; u))=0. \label{nullaver}
	\end{equation}
\end{defn}

The above controllability concepts have corresponding dual formulations in terms of observability for the adjoint system \eqref{sys:adjSchr}.

\begin{defn} 
	System \eqref{sys:adjSchr} is exactly averaged observable in $L^2(G)$ with observability cost $\mathscr{C}_{exob}=\mathscr{C}_{exob}(G,G_0, \xi , T)$ if for all $z_T\in L^2(G)$, the following inequality holds
	\begin{equation}\label{est:exact_obs}
		\left\|z_T\right\|_{L^2(G)}\leq \mathscr{C}_{exob}\int_0^T\int_{G_0}|\mathbb{E}( z(t,x;\xi ;z_T))|\mathrm{d} x \mathrm{d} t.
	\end{equation}
\end{defn}

\begin{defn} 
	System \eqref{sys:adjSchr} is null averaged observable in $L^2(G)$ with observability cost $\mathscr{C}_{ob}=\mathscr{C}_{ob}(G,G_0, \xi , T)$ if for all $z_T\in L^2(G)$, the following inequality holds
	\begin{equation}\label{est:obs}
		\left\|\mathbb{E}(I^{1-\alpha}_{t,T} z(0,\cdot;\xi; z_{T})) \right\|_{L^2(G)}\leq \mathscr{C}_{ob}\int_0^T\int_{G_0}|\mathbb{E}( z(t,x;\xi ;z_T))|\mathrm{d} x \mathrm{d} t.
	\end{equation}
\end{defn}

\begin{prop} \label{Pr: obser_null cont}
	System \eqref{sys:frac Schr} is null averaged (resp. exactly) controllable if and only if the adjoint problem \eqref{sys:adjSchr} is null (resp. exactly) averaged observable. Moreover, the optimal controllability and observability costs are linked by
	\begin{eqnarray}
		\mathscr{C}_{null}=\mathscr{C}_{ob}\quad (\text{resp.}\; \mathscr{C}_{ex}=\mathscr{C}_{exob}). \label{optimal}
	\end{eqnarray}
\end{prop}

% \begin{remark}
	% The result \eqref{optimal} 
	% is a side result obtained within the proof, which, even if we do not use it in this paper, may have some applications in optimal control theory.  
	% \end{remark}

% Proposition \ref{Pr: obser_null cont} was stated in an abstract setting in \cite[Theorem A.2]{lu2016averaged} without proof. However, it might be proved with some adaptation of the Hilbert Uniqueness Method. This method dates back to  \cite{lions1988controlabilite,russell1978controllability}, 
% and, in particular, for the averaged controllability problem, it dates back to \cite[Theorem 3]{zuazua2014averaged}
% and \cite[Theorem A.1]{lu2016averaged}. Since some modifications  from the literature are needed (see Remark \ref{rk:diffproof}),  we sketch the proof, pointing out the main novelties: 

\begin{proof}
	Using the fractional integration by parts \eqref{IBPF}, we can prove that, for a fixed $u\in L^\infty((0,T)\times G_0)$, $y_0\in L^2(G)$ and $z_T\in L^2(G)$ (if necessary, we first consider 
	$z_T\in D((-\Delta)^{\frac{3}{2}})$ and then pass by a density argument): 
	\begin{equation}
		\begin{aligned}
			\left\langle \mathbb{E}(y(T,\cdot;\xi; y_0; u)),z_T\right\rangle_{L^2(G)}&-\left\langle y_0, \mathbb{E}(I^{1-\alpha}_{t,T}z(0,\cdot;\xi;z_T))\right\rangle_{L^2(G)}
			\\&=\mathbb{E} \left[\langle y(T,\cdot;\xi; y_0; u),z_T\rangle-\langle y_0, I^{1-\alpha}_{t,T}z(0,\cdot;\xi;z_T)\rangle\right]
			\\&=
			\mathbb{E} \left[\int_0^T\int_{G_0} u(t,x)\overline{z(t,x;\xi;z_T)}\mathrm{d} x\mathrm{d} t \right]
			\\&
			=\int_0^T\int_{G_0} u(t,x)\overline{\mathbb{E}(z(t,x;\xi;z_T))}\mathrm{d} x\mathrm{d} t. \label{eq: carcontrol}
		\end{aligned} 
	\end{equation}
	
	Let us suppose that \eqref{sys:frac Schr} is null averaged controllable. Let $u\in L^\infty((0,T)\times G_0)$ such that \eqref{est:connull} and \eqref{nullaver} are satisfied. Applying \eqref{eq: carcontrol} to this control, we obtain 
	\begin{equation*}
		-\left\langle y_0, \mathbb{E}(I^{1-\alpha}_{t,T}z(0,\cdot;\xi;z_T))\right\rangle_{L^2(G)}
		=\int_0^T\int_{G_0} u(t,x)\overline{\mathbb{E}(z(t,x;\xi;z_T))}\mathrm{d} x\mathrm{d} t \qquad \forall z_T\in L^2(G).
	\end{equation*}
	Using Hölder inequality and \eqref{est:connull}, we arrive at 
	\begin{eqnarray*}
		\|\mathbb{E}(I^{1-\alpha}_{t,T}z(0,\cdot;\xi;z_T))\|_{L^2(G)}&=&\sup_{\|y_{0}\|\leq 1}\left|\left\langle y_0, \mathbb{E}(I^{1-\alpha}_{t,T}z(0,\cdot;\xi;z_T))\right\rangle_{L^2(G)}\right|\\
		&\leq & \mathscr{C}_{null}\int_0^T\int_{G_0} |\mathbb{E}(z(t,x;\xi;z_T))|\mathrm{d} x\mathrm{d} t \qquad \forall z_T\in L^2(G).
	\end{eqnarray*}
	This shows the observability inequality \eqref{est:obs} with $\mathscr{C}_{ob}\leq \mathscr{C}_{null}$.
	
	Conversely, let us suppose that the observability estimate \eqref{est:obs} is satisfied.
	We consider the following linear subspace $\mathcal{F}$ of $L^{1}((0,T)\times G_0)$:
	\begin{eqnarray*}
		\mathcal{F}:=\left\{\mathds{1}_{G_0}\mathbb{E}(z(\cdot,\cdot;\xi;z_{T})),\quad z_{T}\in L^2(G)\right\},
	\end{eqnarray*}
	and the linear functional $\Phi$ on $\mathcal{F}$ denied by
	\begin{eqnarray*}
		\Phi(\mathds{1}_{G_0}\mathbb{E}(z(\cdot,\cdot;\xi;z_{T}))):=-\left\langle y_{0}, \mathbb{E}(I^{1-\alpha}_{t,T}z(0,\cdot;\xi;z_{T})) \right\rangle_{L^2(G)}.
	\end{eqnarray*}
	Using the observability inequality \eqref{est:obs}, we obtain that $\Phi$ is well defined and bounded on $\mathcal{F}$ with norm satisfying $$\|\Phi\|_{\mathcal{F}^{\prime}}\leq \mathscr{C}_{ob}\|y_{0}\|_{L^2(G)}.$$
	Then, by Hahn–Banach Extension Theorem, we can extend $\Phi$ to a bounded linear functional $\widetilde{\Phi}$ on $L^{1}((0,T)\times G_0)$ having the same norm. The Riesz Representation Theorem yields the existence of $u\in L^{\infty}((0,T)\times G_0)$ such that for any $v\in L^{1}((0,T)\times G_0)$,
	\begin{eqnarray*}
		\widetilde{\Phi}(v)=\left\langle u,v \right\rangle_{L^{\infty}((0,T)\times G_0),\; L^1((0,T)\times G_0)},
	\end{eqnarray*}
	implying that
	\begin{eqnarray*}
		-\left\langle y_{0}, \mathbb{E}(I^{1-\alpha}_{t,T}z(0,\cdot;\xi;z_{T}))\right\rangle_{L^2(G)}=\int_0^T\int_{G_0} u(t,x)\overline{\mathbb{E}(z(t,x;\xi;z_T))}\mathrm{d} x \mathrm{d} t,
	\end{eqnarray*}
	which, combined with \eqref{eq: carcontrol}, yields that the control $u$ satisfies
	\[\left\langle \mathbb{E}(y(T,\cdot;\xi; y_0; u)),z_T\right\rangle_{L^2(G)}=0 \qquad \forall z_T\in L^2(G).\]
	This shows that the average solution takes $0$ at time $T$. Finally, according to the Riesz Representation Theorem, we have
	\begin{equation*}
		\|u\|_{L^\infty((0,T)\times G_0)}=\|\widetilde{\Phi}\|_{(L^1((0,T)\times G_0))^{\prime}}=\|\Phi\|_{\mathcal{F}^{\prime}}\leq \mathscr{C}_{ob}\|y_0\|_{L^2(G)},
	\end{equation*}
	showing that the average solution can be controlled at a cost $\mathscr{C}_{null}\leq \mathscr{C}_{ob}$.\\
	A similar argument applies in the case of exact controllability.
\end{proof}

\begin{rmk}
	The null averaged (resp. exact) observability of \eqref{sys:adjSchr} is equivalent to proving that: there is $C>0$ 
	such that the solution of the adjoint system \eqref{sys:adjSchr} satisfies: for all $z_T\in L^2(G)$, we have
	\begin{equation*}
		\left\|\sum_{n=0}^{\infty} \varphi_{\alpha,1}^{(\xi)}(-\lambda_n T^{\alpha}) \left\langle z_T, e_n\right\rangle e_n\right\|_{L^2(G)}\leq C\int_0^T\int_{G_0}\left|\sum_{n=0}^{\infty} \varphi_{\alpha,\alpha}^{(\xi)}(-\lambda_n (T-t)^{\alpha})) \left\langle z_T, e_n\right\rangle e_n\right|\mathrm{d} x \mathrm{d} t, 
	\end{equation*}
	\begin{equation*}
		\left(\text{resp.} \left\|z_T\right\|_{L^2(G)}\leq C\int_0^T\int_{G_0}\left|\sum_{n=0}^{\infty} \varphi_{\alpha,\alpha}^{(\xi)}(-\lambda_n (T-t)^{\alpha}) \left\langle z_T, e_n\right\rangle e_n\right|\mathrm{d} x \mathrm{d} t\right),
	\end{equation*}
	where $\varphi_{\alpha,\beta}^{(\xi)}$ is the two-parameter fractional characteristic function defined in \eqref{def:fractional charfunc}.
\end{rmk}

\section{Lack of simultaneous null controllability}\label{Sec3}
Simultaneous null controllability fails for any absolutely continuous random variable. In fact, the set of realizations for which such a property holds is negligible, as follows from the next theorem.
\begin{theorem}\label{thm:nonsimcon}
	Let $G\subset\mathbb{R}^d$ be a Lipschitz domain, $G_0\subset G$ be a subset of positive measure, $T>0$, $y_0\in L^2(G)\setminus\{0\}$ and
	$u\in L^\infty((0,T)\times G_0)$. Then, the set:
	\begin{eqnarray*}
		\mathscr{E}:=\left\{\xi\in\mathbb{R}\;:\; y(T,\cdot;\xi;y_0;u)=0\right\}
	\end{eqnarray*}
	is countable.
\end{theorem}
To prove this theorem, we need the following lemmas.
\begin{lem} \label{lm:Muntz}
	Let $\alpha>0$. Then, 
	$\operatorname{span}\{\, t^{\alpha k + \alpha - 1} : k \ge 0 \,\}$ is dense in $L^1(0,T)$.
\end{lem}
\begin{proof}
	The sequence \(\{\alpha k + \alpha - 1\}_{k \ge 0}\) consists of distinct real numbers greater than \(-1\) and $$\displaystyle\sum_{n\geq 0} \frac{\alpha n+\alpha}{\alpha^2n^2+2\alpha^2n+\alpha^2+1}=\infty,$$ so the full Müntz theorem in \(L^1(0,T)\) (see \cite[Theorem 2.3]{borwein1996full}) guarantees that the linear span of 
	\(\{ t^{\alpha k + \alpha - 1} \}_{k \ge 0}\) is dense in \(L^1(0,T)\).
\end{proof}
\begin{lem}\label{lm:analycond}
	Let $\alpha\in (0,1)$, $T>0$ and $f\in L^\infty(0,T)$. Then, the mapping
	\[
	\begin{aligned}
		\mathcal{I} : L^{\infty}(0,T) &\longrightarrow \mathcal{H}(\mathbb{C}) \\
		f &\longmapsto \mathcal{I}(f)(z) := \int_{0}^{T}(T-s)^{\alpha-1}E_{\alpha, \alpha}\left(-\mathrm{i}z (T-s)^\alpha\right)f(s)\;\mathrm{d} s
	\end{aligned}
	\]
	is linear and one-to-one from \(L^{\infty}(0,T)\) into the space of entire functions \(\mathcal{H}(\mathbb{C})\).
\end{lem}
\begin{proof}
	Note that, for any fixed \( z \in \mathbb{C} \), the transformation \(\mathcal{I}(f)(z)\) is well defined due to the boundedness of the function \(f\) and the Mittag--Leffler function, and the condition \(\alpha-1>-1\). By the power series for the Mittag--Leffler function, we obtain 
	\[
	\mathcal I(f)(z)=\int_{0}^{T} \sum_{k=0}^{\infty}\frac{(-\mathrm{i})^{k}}{\Gamma(\alpha k+\alpha)}
	(T-s)^{\alpha k+\alpha-1}f(s)z^{k}\,\mathrm{d}s :=
	\int_{0}^{T}\sum_{k=0}^{\infty} f_k(s)\mathrm{d} s.
	\]
	Since 
	\begin{equation*}
		\sum_{k=0}^{\infty}\int_{0}^{T}|f_k(s)|\mathrm{d} s\le \|f\|_{L^{\infty}(0,T)}T^{\alpha}E_{\alpha,\alpha+1}(T^{\alpha}|z|) <\infty,
	\end{equation*}
	we can exchange the sum and the integral, and thus obtain the entire power series
	\begin{equation}\label{psti}
		\mathcal I(f)(z)=\sum_{k=0}^{\infty} \frac{(-\mathrm{i})^{k}}{\Gamma(\alpha k+\alpha)}
		\Bigg(\int_{0}^{T}(T-s)^{\alpha k+\alpha-1}f(s)\,\mathrm{d}s\Bigg) z^{k}. 
	\end{equation}
	Regarding injectivity, if \(\mathcal{I}(f)\) is identically zero, then all the coefficients of its series vanish, so 
	\[
	\int_0^T t^{\alpha k + \alpha - 1} f(T-t)\,\mathrm{d}t = 0\quad \forall k\in\mathbb{N}.
	\]  
	Since \(f(T-\cdot) \in L^\infty(0,T)\) and vanishes on all these moments, we conclude that \(f(T-\cdot) = 0\) almost everywhere due to Lemma \ref{lm:Muntz}, and therefore \(f = 0\) almost everywhere.
\end{proof}

\begin{rmk}
	When $\alpha = 1$, we recover the classical Paley--Wiener Theorem \cite[Theorem~7.2.1]{strichartz1994guide}.
\end{rmk}

\begin{lem} \label{lm:inteq}Let $\alpha>0$ and $T>0$. There is no function $f \in L^{\infty}(0,T)$ and constant $c \in \mathbb{C}^*$ such that, for every integer $k \ge 0$, one has
	\[
	\int_{0}^{T} (T-s)^{\alpha k + \alpha - 1} f(s)\, \mathrm{d}s
	= c\, T^{\alpha k} \frac{\Gamma(\alpha k + \alpha)}{\Gamma(\alpha k + 1)}.
	\]
	The only possible case is $c=0$ and $f=0$ almost everywhere.
\end{lem}

\begin{proof}
	Since $f \in L^\infty(0,T)$, we have
	\[
	\int_{0}^{T} (T-s)^{\alpha k + \alpha - 1} f(s)\, \mathrm{d}s
	= \mathcal{O}\!\left(\frac{T^{\alpha k + \alpha}}{\alpha k + \alpha}\right), 
	\quad k \to \infty.
	\]
	On the other hand, by \emph{Stirling's formula},
	\[
	\frac{\Gamma(\alpha k + \alpha)}{\Gamma(\alpha k + 1)} 
	\sim (\alpha k)^{\alpha - 1}, 
	\quad k \to \infty.
	\]
	If $c \neq 0$, then the given identity implies
	\[
	k^{\alpha - 1} = \mathcal{O}\!\left(\frac{1}{k}\right), 
	\quad k \to \infty,
	\]
	which is impossible since $\alpha>0$.
	
	If $c = 0$, then, because the family 
	$\{ t^{\alpha k + \alpha - 1} \}_{k \ge 0}$ is dense in $L^1(0,T)$ (see Lemma \ref{lm:Muntz}) and $f \in L^\infty(0,T)$, 
	we necessarily have $f = 0$ almost everywhere.
	
	Hence, the only possible solution is $c=0$ and $f=0$ almost everywhere.
\end{proof}

\begin{proof}[Proof of Theorem \ref{thm:nonsimcon}]
	Using the fractional generalization of Duhamel’s Principle, we obtain
	\begin{align*}
		\mathscr{E}=& \bigcap_{n\in\mathbb{N}}\left\{ \xi\in\mathbb{R}\;:\; \mathcal{I}(f_n)(\xi\lambda_n)=-\langle y_0, e_n\rangle E_{\alpha, 1}\left(-\mathrm{i}\xi \lambda_n T^\alpha\right) \right\}\\
		=&\bigcap_{n\in\mathbb{N}}\left\{ \xi\in\mathbb R\;:\; F_n(\xi\lambda_n)=0 \right\}:=  \bigcap_{n\in\mathbb{N}}\mathscr{E}_{n},
	\end{align*}
	where
	\begin{align*}
		f_n(s):=&\langle\mathds{1}_{G_0}u(s,\cdot),e_n \rangle,\quad s\in (0,T),\\
		F_n(z):=&\mathcal{I}(f_n)(z) + \langle y_0, e_n\rangle E_{\alpha, 1}\left(-\mathrm{i}z T^\alpha\right),\quad z\in\mathbb{C}.
	\end{align*} 
	Hence, it suffices to show that there exists $n \in \mathbb{N}$ such that $\mathscr{E}_n$ is countable.  
	Since $\mathscr{E}_n$ consists of the real zeros of $F_n$ divided by $\lambda_n$, and by Lemma \ref{lm:analycond}, we know that $F_n$ is analytic on $\mathbb{C}$, it is enough to prove that there exists $n \in \mathbb{N}$ such that $F_n \neq 0$.  
	
	Suppose, for the sake of contradiction, that $F_n = 0$ for all $n \in \mathbb{N}$.  
	By identifying the power series of $\mathcal{I}(f_n)$ with that of the Mittag--Leffler function, we obtain
	\begin{align*}
		\int_{0}^{T}(T-s)^{\alpha k+\alpha-1}f_n(s)\,\mathrm{d}s=-\langle y_0,e_n\rangle\frac{T^{\alpha k} \Gamma(\alpha k+\alpha)}{\Gamma(\alpha k+1)}\quad \forall k\in\mathbb{N}.
	\end{align*}
	Hence, by Lemma \ref{lm:inteq}, we deduce that $\langle y_0, e_n \rangle = 0$ for all $n \in \mathbb{N}$, which contradicts the assumption that $y_0 \neq 0$. In conclusion, we obtain that $\mathscr{E}_n$ is countable.
\end{proof}

\section{Lack of exact averaged controllability}\label{Sec4}
The next result concerns the lack of exact averaged controllability of system \eqref{sys:frac Schr}.
\begin{theorem} \label{tm:noexcon} 
	Let $\xi$ be a random variable in $(\Omega, \mathcal{F}, \mathbb{P})$ such that its \textbf{FCF} $\varphi_{\alpha, \alpha}^{(\xi)}$ vanishes at infinity. Let $G\subset\mathbb R^d$ be a Lipschitz domain and $G_0\subset G$ a subset of positive measure. Then, system \eqref{sys:frac Schr} is not exactly averaged controllable in $L^2(G)$ with controls acting in $L^\infty((0,T)\times G_0)$.
\end{theorem}
The proof is based on Lemma \ref{lm:van der Corput}.
\begin{proof}[Proof of Theorem \ref{tm:noexcon}]
	Let us suppose that there is $C>0$ such that the observability inequality \eqref{est:exact_obs} holds, and let us try to obtain a contradiction. Consider as final values $z_T$ the eigenfunctions $e_n$ for $n\in\mathbb{N}$, which satisfy \(\|e_n\|=1.\)
	Then, assuming the observability inequality \eqref{est:exact_obs}, using \eqref{sol:zsum}, for such final values we obtain
	\begin{align} 
		1 &\leq C\int_0^T\int_{G_0}  |\mathbb{E}(z(t,x;\xi; e_n))| \mathrm{d} x\mathrm{d} t \nonumber\\
		&\leq C\int_0^T\int_{G}  |\mathbb{E}(z(t,x;\xi;e_n))| \mathrm{d} x \mathrm{d} t \nonumber\\
		&=C\int_0^T\int_{G}  |(T-t)^{\alpha-1}\varphi_{\alpha, \alpha}^{(\xi)}(-\lambda_{n}(T-t)^{\alpha})e_n(x)| \mathrm{d} x \mathrm{d} t  \nonumber\\
		&\leq C \sqrt{|G|}\int_0^T (T-t)^{\alpha-1}|\varphi_{\alpha, \alpha}^{(\xi)}(-\lambda_{n}(T-t)^{\alpha})| \mathrm{d} t \nonumber\\
		& =C \sqrt{|G|}\int_0^T t^{\alpha-1}|\varphi_{\alpha, \alpha}^{(\xi)}(-\lambda_{n}t^{\alpha})| \mathrm{d} t. \label{est:eq}
	\end{align}
	Since $\varphi_{\alpha, \alpha}^{(\xi)}(-\lambda_{n}t^{\alpha})$ converges to $0$ as $n\rightarrow \infty$ for all $t\in (0,T]$. Moreover, by Lemma \ref{lm:van der Corput}, $$t^{\alpha-1}|\varphi_{\alpha, \alpha}^{(\xi)}(-\lambda_{n}t^{\alpha})|\leq C_0 t^{\alpha-1}\in L^1(0,T).$$
	Therefore, the Lebesgue dominated convergence theorem implies that the integral \eqref{est:eq} tends to $0$ as $n\rightarrow \infty$. This leads to a contradiction, proving the failure of \eqref{est:exact_obs}.
\end{proof}

Lemma \ref{lm:van der Corput} implies that $\varphi_{\alpha, \alpha}^{(\xi)}$ vanishes at infinity for any absolutely continuous random variable $\xi$ in $(\Omega, \mathcal{F}, \mathbb{P})$. This leads to the following corollary.
\begin{cor} \label{cor:noexc}
	Let $\xi$ be an absolutely continuous random variable in $(\Omega, \mathcal{F}, \mathbb{P})$. Then the system \eqref{sys:frac Schr} is not exactly averaged controllable in $L^2(G)$ with controls acting in $L^\infty((0,T)\times G_0)$.
\end{cor}
\begin{rmk}
	The system \eqref{sys:frac Schr} may fail to be exactly averaged controllable even when the distribution is discrete, as illustrated by the Rademacher distribution (see Proposition \ref{prop:FCFRade}). 
\end{rmk}

%\begin{remark}
%A similar proof would work for controls in $L^2((0,T)\times G_0)$.
%\end{remark}

\section{Null averaged controllability}\label{Sec5}
The averaged observability inequality is characterized by the properties of the fractional characteristic function. Following a detailed analysis, we propose a class of real random variables for which the null averaged controllability holds. 

\begin{defn} Let $\alpha\in(0,1)$ and $T>0$. 
	A real random variable $\xi$ belongs to the class $\mathcal{C}_\alpha$ if its fractional characteristic function $\varphi_{\alpha,1}^{(\xi)}$ satisfies: there are constants $c,\lambda^{\star},\theta,C_1,C_2>0$, $r>\frac{1}{2}$, and $\delta\in (0, 1)$
	such that:
	\begin{align}
		&  |\varphi_{\alpha,1}^{(\xi)}(-\lambda t^\alpha)|\leq e^{-c\lambda^r(t-s)^{\theta}}|\varphi_{\alpha,1}^{(\xi)}(-\lambda s^\alpha)|,\quad \delta T\le s<t \le T\;\;\mbox{and}\;\; \lambda \ge\lambda^\star, \label{eq:hypFCF1}\\
		& C_1|\varphi_{\alpha,1}^{(\xi)}(-\lambda t^\alpha)|\leq |\varphi_{\alpha,\alpha}^{(\xi)}(-\lambda t^\alpha)|\leq C_2|\varphi_{\alpha,1}^{(\xi)}(-\lambda t^\alpha)|, \quad \delta T\le t \le T\;\;\mbox{and}\;\;  \lambda \ge\lambda^\star. \label{eq:hypFCF2}
	\end{align}
\end{defn}

An interesting example of such a random variable is given by the Rademacher distribution, defined in Example \ref{Ex: Rademacher distribution}.
\begin{prop}\label{propR}
	The Rademacher distribution belongs to the class $\mathcal{C}_\alpha$ for all $\alpha\in (0,\frac{1}{2})$.
\end{prop}
\begin{proof} Let $T>0$. 
	First, we have $\varphi_{\alpha,\beta}^{(\xi)}(s) = E_{2\alpha,\beta}(-s^2),\; s\in\mathbb{R}$ (see Proposition \ref{prop:FCFRade}).

	From \cite[Lemma 2.3]{xue2017observability}, for any $\delta\in (0,1)$, there exist $A,B,\lambda^{\star}>0$ depending only on $\delta$, $T$ such that:
	\[ AE_{2\alpha,1}(-\lambda t^{2\alpha})\le \frac{\sin(2\alpha \pi)}{\pi} \int_{1}^{+\infty} 
	\frac{\zeta^{2\alpha-1} e^{-\zeta \lambda^{1/(2\alpha)} t}}{1 + 2 \zeta^{2\alpha} \cos(2\alpha \pi) + \zeta^{4\alpha}} \, \mathrm{d}\zeta \le B E_{2\alpha,1}(-\lambda t^{2\alpha}),\]
	for all $t\in [\delta T,T]$ and $\lambda>\lambda^\star$.
	
	Let $s<t$ such that $s,t\in [\delta T,T]$, then
	\begin{align*}
		|\varphi_{\alpha,1}^{(\xi)}(-\lambda t^\alpha)|&=E_{2\alpha,1}(-\lambda^2 t^{2\alpha})\\
		&\le \frac{\sin(2\alpha \pi)}{A\pi} \int_{1}^{+\infty} 
		\frac{\zeta^{2\alpha-1} e^{-\zeta \lambda^{1/\alpha} t}}{1 + 2 \zeta^{2\alpha} \cos(2\alpha \pi) + \zeta^{4\alpha}} \, \mathrm{d}\zeta\\
		&\le \frac{\sin(2\alpha \pi)}{A\pi} e^{-\lambda^{1/\alpha}(t-s)}\int_{1}^{+\infty} 
		\frac{\zeta^{2\alpha-1} e^{-\zeta \lambda^{1/\alpha} s}}{1 + 2 \zeta^{2\alpha} \cos(2\alpha \pi) + \zeta^{4\alpha}} \, \mathrm{d}\zeta\\
		&\le \frac{B}{A} e^{-\lambda^{1/\alpha}(t-s)}E_{2\alpha,1}(-\lambda^2 s^{2\alpha})\\
		&= \frac{B}{A} e^{-\lambda^{1/\alpha}(t-s)}|\varphi_{\alpha,1}^{(\xi)}(-\lambda s^\alpha)|\\
		&\le e^{-c\lambda^{1/\alpha}(t-s)}|\varphi_{\alpha,1}^{(\xi)}(-\lambda s^\alpha)|
	\end{align*}
	for some $c>0$ depending only on $\delta$ and $T$. This shows \eqref{eq:hypFCF1} for $r=\frac{1}{\alpha}>1$ and $\theta=1$.
	
	Regarding \eqref{eq:hypFCF2}, we use the asymptotic expansions \eqref{asymptotic expansions} of Mittag--Leffler functions, for any $t\in [\delta T,T]$ we have 
	\[ E_{2\alpha,\beta}(-\lambda^2 t^{2\alpha})\sim \frac{1}{\Gamma(\beta-2\alpha)}\frac{1}{\lambda^2 t^{2\alpha}},\quad \lambda\to \infty.\]
	Then
	\[\frac{E_{2\alpha,\alpha}(-\lambda^2 t^{2\alpha})}{E_{2\alpha,1}(-\lambda^2 t^{2\alpha})}\to \frac{\Gamma(1-2\alpha)}{\Gamma(-\alpha)}:=-r_\alpha <0,\quad \mbox{as}\;\;\lambda\to \infty, \;\mbox{uniformly for } t \in (\delta T, T)
	.\]
	Hence, by taking $\lambda^{\star}>0$ large enough:
	\[0<\frac{r_\alpha}{2}\le \frac{|E_{2\alpha,\alpha}(-\lambda^2 t^{2\alpha})|}{|E_{2\alpha,1}(-\lambda^2 t^{2\alpha})|}\le 2r_\alpha,\quad t\in(\delta T,T)\;\;\mbox{and}\;\;\lambda \ge\lambda^\star,\]
	which yields \eqref{eq:hypFCF2}.
\end{proof}

%\begin{remark}
%In the case of the Rademacher distribution, \eqref{eq:hypFCF1} and \eqref{eq:hypFCF2} are also true with $0\leq s<t\leq T$ and $0\leq t\leq T$ respectively. This can be shown by making $A$ and $B$ depending on $\alpha$ and using compactness. 
%\end{remark}

\begin{rmk}
	The Rademacher distribution does not belong to the class $\mathcal{C}_{\frac{1}{2}}$. Indeed, for $\alpha=\frac{1}{2}$, we have
	$$\varphi_{\alpha,\alpha}^{(\xi)}(-\lambda t^\alpha)=E_{1,\frac{1}{2}}(-\lambda^2 t), \qquad \varphi_{\alpha,1}^{(\xi)}(-\lambda t^\alpha)=E_{1,1}(-\lambda^2 t)=e^{-\lambda^2 t}.$$
	Then, by \eqref{asymptotic expansions} applied to $E_{1,\frac{1}{2}}(-\lambda^2t)$, we obtain
	\begin{align*}
		\frac{|\varphi_{\alpha,\alpha}^{(\xi)}(-\lambda t^\alpha)|}{|\varphi_{\alpha,1}^{(\xi)}(-\lambda t^\alpha)|}& \sim \frac{\frac{1}{2\sqrt{\pi}}\frac{1}{\lambda^2 t}}{e^{-\lambda^2 t}}\to \infty, \qquad \text{as }\lambda\to \infty.
	\end{align*}
	Therefore, \eqref{eq:hypFCF2} cannot hold. However, the corresponding average system, which consists of the following biharmonic heat equation
	\begin{equation*} 
		\begin{cases}
			\partial_{t} \tilde{y} + \Delta^2 \tilde{y}=\mathds{1}_{G_0} u,& \mbox{ on }Q_T,\\
			\Delta\tilde{y}=\tilde{y}=0, & \mbox{ on }\Sigma_T,\\
			\tilde{y}(0,\cdot)=y_0, & \mbox{ in } G,
		\end{cases}
	\end{equation*}
	is null controllable at any $T>0$; see \cite{guerrero2019}. Thus, being in $\mathcal C_\alpha$ is not a necessary condition. In fact, finding a characterization is an open problem. 
\end{rmk}

%\begin{remark}
%Note that in the case of the integer derivative (i.e., $\alpha=1$), the class $\mathcal{C}_1$ includes the normal, Cauchy as well as stable distributions; see \cite{barcena2025}.
%\end{remark}

Our main result on the null averaged controllability of the fractional Schrödinger system \eqref{sys:frac Schr} reads as follows:
\begin{theorem} \label{thm:nulaver}
	Let $G \subset \mathbb{R}^{d}$ be a Lipschitz locally star-shaped domain, $G_0 \subset G$ be a subset of positive measure, $\alpha\in (0,1)$, $T>0$ and $\xi$ a real random variable belonging to the class $\mathcal{C}_\alpha$. Then, the system \eqref{sys:frac Schr} is null averaged controllable at time $T$.
\end{theorem}
The proof of Theorem \ref{thm:nulaver} is inspired on the proof for the time-fractional diffusion equation in \cite{xue2017observability}, which follows the
spectral approach in \cite{apraiz2014observability, miller2010observability}, which are inspired by \cite{lebeau1995controle}.

Let us denote $\Lambda_\lambda := \{n:\lambda_n\leq \lambda\}$ (resp. $\Lambda_\lambda^{\perp} := \{n:\lambda_n> \lambda\}$) for all $\lambda > 0$, and by $\mathcal{P}_\lambda$ (resp. $\mathcal{P}^\perp_\lambda$) the orthogonal projection of $L^2(G)$ onto $\langle e_n \rangle_{n \in \Lambda_\lambda}$ (resp. $\langle e_n \rangle_{n \in \Lambda_\lambda^\perp}$). 

The following spectral estimate for general subsets of positive measure is an immediate consequence of \cite[Lemma 3.2]{barcena2025}, followed by the application of Theorems 3 and 5 in \cite{apraiz2014observability} for the observation set $B(x_0,r)\cap G_0$.
\begin{lem}\label{lm:obsellG0}
	Let $G\subset \mathbb{R}^{d}$ be a Lipschitz locally star-shaped domain, $G_0\subset G$ be a subset of positive measure, and $\{e_j\}$ be the orthonormal eigenfunctions of the Dirichlet Laplacian. Then, there exists a constant $C>0$ such that for all $\lambda>0$ and $\{c_j\}\subset\mathbb{C}$:
	\begin{equation}\label{est:seqobs}
		\left(\sum_{j\in\Lambda_\lambda}|c_j|^2\right)^{1/2}\leq C
		e^{C\sqrt \lambda}\left\|\sum_{j\in\Lambda_\lambda} c_je_j
		\right\|_{L^1(G_0)}.
	\end{equation}
\end{lem}
In fact, in \cite{apraiz2014observability} they prove it for real values, but by writing $c_j=a_j+ib_j$, with $a_j,b_j\in\mathbb{R}$, using the estimate for both $a_j$ and $b_j$, and the estimates $\Big|\sum_{j\in\Lambda_\lambda}a_je_j\Big|\leq \Big|\sum_{j\in\Lambda_\lambda}c_je_j\Big|$ and $\Big|\sum_{j\in\Lambda_\lambda}b_je_j\Big|\leq \Big|\sum_{j\in\Lambda_\lambda}c_je_j\Big|$ since the eigenfunctions $e_j$ are real-valued. We may easily obtain the complex case. 
\begin{rmk} \label{scaling}
	To apply \eqref{eq:hypFCF1} and \eqref{eq:hypFCF2} with $\lambda=\lambda_j$, the Dirichlet Laplacian eigenvalues, we may assume without loss of generality that, after a spatial scaling, $\lambda_j\ge\lambda^\star$.
	Indeed, for fixed $k=\sqrt{\frac{\lambda^\star}{\lambda_0}}$, if we set $y_k(t,x)=y(t,\frac{x}{k})$, then $y$ solves \eqref{sys:frac Schr} if and only if $y_k$ solves 
	\begin{equation}\label{sys:scaling}  
		\begin{cases}
			\partial^{\alpha}_{0,t} y_k -\xi \mathrm{i} k^2 \Delta y_k=\mathds{1}_{G_0^k} u_k,& \mbox{ on }Q^k_T,\\
			y_k=0, & \mbox{ on }\Sigma^k_T,\\
			y_k(0,\cdot)=y_{0,k}, & \mbox{ in } G^k,
		\end{cases}
	\end{equation}
	where $G^k=k.G,\; G_0^k=k.G_0,\; Q^k_T=(0,T)\times G^k,\; \Sigma^k_T=(0,T)\times \partial G^k, \; u_k(t,\cdot)=u(t,\frac{\cdot}{k})$ and $y_{0,k}=y_0(\frac{\cdot}{k})$. In this case, the eigenvalues of $-k^{2}\Delta$ are $k^{2}\lambda_{j}\ge k^{2}\lambda_{0}=\lambda^\star$. Thus, we can proceed with the system \eqref{sys:scaling}, to which we may apply \eqref{eq:hypFCF1} and \eqref{eq:hypFCF2} using the new eigenvalues and we conclude from the fact that \eqref{sys:frac Schr} is null averaged controllable if and only if \eqref{sys:scaling} is null averaged controllable. 
\end{rmk}
\begin{proof}[Proof of Theorem \ref{thm:nulaver}]
	For clarity, the proof is divided into several steps.
	
	\textbf{Step 1.} We aim to prove that there are $C>0$ large enough and $c>0$ small enough such that for all
	$z_T\in L^2(G)$, $\lambda\ge \lambda_{0}$ (as explained in Remark \ref{scaling}, we may assume that $\lambda_0 \ge \lambda^\star$) 
	$t,s\in [\delta T,T]$ such that $s<t$, and $\tau\in\left[\frac{s+t}{2},t\right]$:
	\begin{align}
		&\|\mathbb{E}(I^{1-\alpha}_{t,T} z(T-t,\cdot;\xi; z_{T}))\|_{L^2(G)} \nonumber\\
		&\le Ce^{C\sqrt{\lambda}}\left(\|\mathbb{E}( z(T-\tau,\cdot;\xi; z_{T}))\|_{L^1(G_0)} + e^{-c\lambda^r(t-s)^\theta}\| \mathbb{E}(I^{1-\alpha}_{t,T} z(T-s,\cdot;\xi; z_{T})) \|_{L^2(G)} \right). \label{E1}
	\end{align} 
	It will be important for later on that $\lambda\ge \lambda_{0}$ can be chosen as a function of $s$ and $t$ (see \eqref{app lambda}); for this purpose, $c$ will be required, among other restrictions, that: 
	\begin{equation}\label{est:cthetalambda*}
		c\le \frac{1}{T^\theta\lambda_{0}^r}.
	\end{equation}

	First, using \eqref{eq:hypFCF1} and \eqref{eq:hypFCF2}, we have:
	\begin{equation} \label{E11}
		\begin{split}
			\|\mathbb{E}(I^{1-\alpha}_{t,T} z(T-t,\cdot;\xi; z_{T}))\|_{L^2(G)}^2 &=\sum_{j\in\mathbb N}|\langle z_T,e_j\rangle|^{2}|\varphi_{\alpha,1}^{(\xi)}(-\lambda_{j}t^\alpha)|^2 \\
			&=\sum_{j\in \Lambda_\lambda}|\langle z_T,e_j\rangle|^{2}|\varphi_{\alpha,1}^{(\xi)}(-\lambda_{j}t^\alpha)|^2+ \sum_{j\in \Lambda_\lambda^\perp}|\langle z_T,e_j\rangle|^{2}|\varphi_{\alpha,1}^{(\xi)}(-\lambda_{j}t^\alpha)|^2 \\
			&\le \sum_{j\in \Lambda_\lambda}|\langle z_T,e_j\rangle|^{2}|\varphi_{\alpha,1}^{(\xi)}(-\lambda_{j}\tau^\alpha)|^2+ \sum_{j\in \Lambda_\lambda^\perp}|\langle z_T,e_j\rangle|^{2}|\varphi_{\alpha,1}^{(\xi)}(-\lambda_{j}t^\alpha)|^2 	\\
			&\le T^{2(1-\alpha)}C_1^{-2}\sum_{j\in \Lambda_\lambda}|\langle z_T,e_j\rangle|^{2}\tau^{2(\alpha-1)}|\varphi_{\alpha,\alpha}^{(\xi)}(-\lambda_{j}\tau^\alpha)|^2 \\
			&\quad + e^{-c\lambda^r(t-s)^{\theta}}\sum_{j\in \Lambda_\lambda^\perp}|\langle z_T,e_j\rangle|^{2}|\varphi_{\alpha,1}^{(\xi)}(-\lambda_{j}s^\alpha)|^2 \\
			&\le  T^{2(1-\alpha)}C_1^{-2}\sum_{j\in \Lambda_\lambda}|\langle z_T,e_j\rangle|^{2}\tau^{2(\alpha-1)}|\varphi_{\alpha,\alpha}^{(\xi)}(-\lambda_{j}\tau^\alpha)|^2 \\
			&\quad + e^{-c\lambda^r(t-s)^{\theta}}\|\mathbb{E}(I^{1-\alpha}_{t,T} z(T-s,\cdot;\xi; z_{T}))\|^2_{L^2(G)}. 
		\end{split}
	\end{equation}
	On the other hand, using Lemma \ref{lm:obsellG0}, \eqref{sol:zsum} and $\mathcal{P}_\lambda z_T=z_T -\mathcal{P}_\lambda^\perp z_T$, we obtain:
	\begin{equation} \label{E12}
		\begin{split}
			\sum_{j\in \Lambda_\lambda}|\langle z_T,e_j\rangle|^{2}\tau^{2(\alpha-1)}|\varphi_{\alpha,\alpha}^{(\xi)}(-\lambda_{j}\tau^\alpha)|^2 
			&\le  C^2e^{2C\sqrt{\lambda}} \Big\| \sum_{j\in \Lambda_\lambda}\langle z_T,e_j\rangle \tau^{\alpha-1}\varphi_{\alpha,\alpha}^{(\xi)}(-\lambda_{j}\tau^\alpha)e_j\Big\|^2_{L^1(G_0)}\\
			&=C^2e^{2C\sqrt{\lambda}}\|\mathbb{E}( z(T-\tau,\cdot;\xi; \mathcal{P}_\lambda z_{T}))\|^2_{L^1(G_0)}\\
			&\le 2C^2e^{2C\sqrt{\lambda}}\|\mathbb{E}( z(T-\tau,\cdot;\xi; z_{T}))\|^2_{L^1(G_0)} \\
			&\quad + 2C^2e^{2C\sqrt{\lambda}}\|\mathbb{E}( z(T-\tau,\cdot;\xi; \mathcal{P}_\lambda^\perp z_{T}))\|^2_{L^1(G_0)}.
		\end{split}
	\end{equation}
	Now, the Cauchy–Schwarz inequality together with \eqref{sol:zsum}, $|\tau-s|\geq \frac{|t-s|}{2}$ and the estimates \eqref{eq:hypFCF1} and \eqref{eq:hypFCF2} yield:
	\begin{equation} \label{E13}
		\begin{split}
			\|\mathbb{E}( z(T-\tau,\cdot;\xi; \mathcal{P}_\lambda^\perp z_{T}))\|^2_{L^1(G_0)}  &\le |G_0|\|\mathbb{E}( z(T-\tau,\cdot;\xi; \mathcal{P}_\lambda^\perp z_{T}))\|^2_{L^2(G_0)}\\
			&\le  |G_0|\|\mathbb{E}(z(T-\tau,\cdot;\xi; \mathcal{P}_\lambda^\perp z_{T}))\|^2_{L^2(G)} \\
			&= |G_0|\sum_{j\in \Lambda_\lambda^\perp}|\langle z_T,e_j\rangle|^{2}\tau^{2(\alpha-1)}|\varphi_{\alpha,\alpha}^{(\xi)}(-\lambda_{j}\tau^\alpha)|^2\\
			&\le  (\delta T)^{2(\alpha-1)} C_2^{2} |G_0|\sum_{j\in \Lambda_\lambda^\perp}|\langle z_T,e_j\rangle|^{2}|\varphi_{\alpha,1}^{(\xi)}(-\lambda_{j}\tau^\alpha)|^2\\
			&\le  (\delta T)^{2(\alpha-1)} C_2^{2} |G_0|e^{-2^{-\theta}c\lambda^r(t-s)^\theta}\sum_{j\in \Lambda_\lambda^\perp}|\langle z_T,e_j\rangle|^{2}|\varphi_{\alpha,1}^{(\xi)}(-\lambda_{j}s^\alpha)|^2\\
			&\le  (\delta T)^{2(\alpha-1)} C_2^{2} |G_0|e^{-2^{-\theta}c\lambda^r(t-s)^\theta}\| \mathbb{E}(I^{1-\alpha}_{t,T} z(T-s,\cdot;\xi; z_{0})) \|^2_{L^2(G)}.
		\end{split}
	\end{equation}
	In conclusion, by combining \eqref{E11}, \eqref{E12}, and \eqref{E13}, and taking 
	$C>0$ sufficiently large and $c>0$ sufficiently small, we obtain \eqref{E1}.
	
	\textbf{Step 2.} We now prove that for all $\kappa \in (0,1)$, there exists a sufficiently large constant $C > 0$, depending on $\kappa$, such that the following interpolation estimate holds:
	\begin{align}
		&\|\mathbb{E}(I^{1-\alpha}_{t,T} z(T-t,\cdot;\xi; z_{T}))\|_{L^2(G)}\nonumber\\
		&\leq\left[C e^{C (t-s)^{-\sigma}} \|\mathbb{E}( z(T-\tau,\cdot;\xi; z_{T}))\|_{L^1(G_0)}\right]^{1-\kappa} \|\mathbb{E}(I^{1-\alpha}_{t,T} z(T-s,\cdot;\xi; z_{T}))\|_{L^2(G)}^\kappa, \label{E2}
	\end{align}
	for all $s,t\in [\delta T,T]$ such that $s<t$, $\tau\in[\frac{s+t}{2},t]$
	and $\sigma=\frac{\theta}{2r-1}$.
	
	In what follows, we assume $
	\|\mathbb{E}(z(T-\tau,\cdot;\xi; z_{T}))\|_{L^1(G_0)}\neq 0$ and $\|\mathbb{E}(I^{1-\alpha}_{t,T} z(T-s,\cdot;\xi; z_{T}))\|_{L^2(G)}\neq 0.$ Indeed, if $\|\mathbb{E}( z(T-\tau,\cdot;\xi; z_{T}))\|_{L^1(G_0)}=0$, by taking $\lambda\to\infty$ in \eqref{E1}, we obtain that $\|\mathbb{E}(I^{1-\alpha}_{t,T} z(T-t,\cdot;\xi; z_{T}))\|_{L^2(G)}=0$, so \eqref{E2} holds. In addition, if $\|\mathbb{E}(I^{1-\alpha}_{t,T}z(T-s,\cdot;\xi; z_{T}))\|_{L^2(G)}=0$, then by \eqref{sol:zsumrev} and \eqref{eq:hypFCF1}, we obtain  $\|\mathbb{E}(I^{1-\alpha}_{t,T} z(T-t,\cdot;\xi; z_{T}))\|_{L^2(G)}=0,$ so \eqref{E2} also holds. 
	
	Since $r>\frac{1}{2}$, we can balance the two exponential exponents in \eqref{E1} via the weight $\kappa\in (0,1)$ by observing that:
	\begin{align*}
		\max_{\lambda>0}\left(C\sqrt{\lambda}-c\kappa\lambda^r(t-s)^\theta\right)=&\frac{2r-1}{(2r)^{\frac{2r}{2r-1}}}C^{\frac{2r}{2r-1}}(c\kappa)^{-\frac{1}{2r-1}}(t-s)^{-\frac{\theta}{2r-1}}\\
		\le& C^\prime (t-s)^{-\frac{\theta}{2r-1}}=C^\prime (t-s)^{-\sigma},
	\end{align*}
	for some suitable constant $C^\prime > 0$ 
	depending on $c$, $C$, $r$, $\kappa$, 
	and $\sigma=\frac{\theta}{2r-1}$. Hence, \eqref{E1} yields
	\begin{equation} \label{E21}
		\begin{split}
			\|\mathbb{E}(I^{1-\alpha}_{t,T} z(T-t,\cdot;\xi; z_{T}))\|_{L^2(G)} 
			& \le  Ce^{C^\prime (t-s)^{-\sigma}}
			\left(e^{c\kappa\lambda^r(t-s)^\theta}\|\mathbb{E}(z(T-\tau,\cdot;\xi; z_{T}))\|_{L^1(G_0)} \right. \\
			&\phantom{\le  Ce^{C^\prime (t-s)^{-\sigma}}(e}
			\left.+ e^{c(\kappa-1)\lambda^r(t-s)^\theta}\| \mathbb{E}(I^{1-\alpha}_{t,T} z(T-s,\cdot;\xi; z_{T})) \|_{L^2(G)} \right)\\
			&=Ce^{C^\prime (t-s)^{-\sigma}}
			\left[\Big(e^{c\lambda^r(t-s)^\theta}\Big)^\kappa\|\mathbb{E}(z(T-\tau,\cdot;\xi; z_{T}))\|_{L^1(G_0)} \right. 
			\\&\phantom{Ce^{C^\prime (t-s)^{-\sigma}}(e}
			\left.	+ \Big(e^{c\lambda^r(t-s)^\theta}\Big)^{\kappa-1}\| \mathbb{E}(I^{1-\alpha}_{t,T} z(T-s,\cdot;\xi; z_{T})) \|_{L^2(G)} \right].
		\end{split}
	\end{equation}
	To determine the optimal bound on the right-hand side of \eqref{E21}, we consider the
	function 
	\[x\mapsto x^\kappa\|\mathbb{E}( z(T-t,\cdot;\xi; z_{T}))\|_{L^1(G_0)} + x^{\kappa-1}\| \mathbb{E}(I^{1-\alpha}_{t,T} z(T-s,\cdot;\xi; z_{T})) \|_{L^2(G)},\] defined on $(0,\infty)$, which admits a unique minimum at
	\[
	x_*=\frac{(1-\kappa)\mathbb{E}(I^{1-\alpha}_{t,T} z(T-s,\cdot;\xi; z_{T})) \|_{L^2(G)}}{\kappa \|\mathbb{E}(z(T-\tau,\cdot;\xi; z_{T}))\|_{L^1(G_0)}}.
	\]
	Then, we discuss two possible cases:
	
	\underline{If $x_*< e$.} In this case, we have that: 
	\begin{equation} \label{E22}
		\begin{split}
			\|\mathbb{E}(I^{1-\alpha}_{t,T} z(T-s,\cdot;\xi; z_{T}))\|_{L^2(G)} &\le e\left(\frac{\kappa}{1-\kappa}\right)\|\mathbb{E}( z(T-\tau,\cdot;\xi; z_{T}))\|_{L^1(G_0)}\\
			& \leq C^{\prime\prime} e^{C^{\prime\prime} (t-s)^{-\sigma}} \|\mathbb{E}( z(t,\cdot;\xi; z_{T}))\|_{L^1(G_0)},
		\end{split}
	\end{equation}
	for some sufficiently large constant $C^{\prime\prime}>0$ depending on $\kappa$.
	
	Using \eqref{eq:hypFCF1}, \eqref{sol:zsum} and the estimate \eqref{E22}, we obtain:
	\begin{equation*}
		\begin{split}
			\|\mathbb{E}(&I^{1-\alpha}_{t,T} z(T-t,\cdot;\xi; z_{T}))\|_{L^2(G)}\leq \|\mathbb{E}(I^{1-\alpha}_{t,T} z(T-s,\cdot;\xi; z_{T}))\|_{L^2(G)}\\
			&\leq \left[C^{\prime\prime} e^{C^{\prime\prime} (t-s)^{-\sigma}} \|\mathbb{E}(z(T-\tau,\cdot;\xi; z_{0}))\|_{L^1(G_0)}\right]^{1-\kappa} \|\mathbb{E}(I^{1-\alpha}_{t,T} z(T-s,\cdot;\xi; z_{T}))\|_{L^2(G)}^\kappa.
		\end{split}
	\end{equation*}

	\underline{If $x_* \ge e$.} 
	Choosing appropriately $\lambda$ given by
	\begin{equation}
		\lambda=\left(\frac{\ln(x_*)}{c(t-s)^\theta}\right)^{\frac{1}{r}} \ge \lambda_{0}. \label{app lambda}
	\end{equation}
	Note that $\lambda \ge \lambda_{0}$ due to \eqref{est:cthetalambda*} and $x_* \ge e$. Then, we can ensure that \(e^{c\lambda^r(t-s)^\theta}=x_*.\) For this choice, \eqref{E21} yields:
	\begin{equation*}
		\begin{split}
			&\|\mathbb{E}(I^{1-\alpha}_{t,T} z(T-t,\cdot;\xi; z_{T}))\|_{L^2(G)}  \\
			&\le  C\left[\left(\frac{1-\kappa}{\kappa}\right)^\kappa+ \left(\frac{1-\kappa}{\kappa}\right)^{k-1}\right]e^{C^\prime (t-s)^{-\sigma}}
			\|\mathbb{E}(z(T-t,\cdot;\xi; z_{T}))\|_{L^1(G_0)}^{1-\kappa}\| \mathbb{E}(I^{1-\alpha}_{t,T} z(T-s,\cdot;\xi; z_{T})) \|_{L^2(G)}^{\kappa}
			\\ & \le  \left[C^{\prime\prime} e^{C^{\prime\prime} (t-s)^{-\sigma}}
			\|\mathbb{E}( z(T-t,\cdot;\xi; z_{T}))\|_{L^1(G_0)}\right]^{1-\kappa}\| \mathbb{E}(I^{1-\alpha}_{t,T} z(T-s,\cdot;\xi; z_{T})) \|_{L^2(G)}^{\kappa},
		\end{split}
	\end{equation*}
	for some sufficiently large constant $C^{\prime\prime}>0$ depending on $\kappa$.

	Consequently, in all cases, the estimate \eqref{E2} is fulfilled.
	
	\textbf{Step 3.} Now, we are in a position to derive the observability inequality \eqref{est:obs} from the interpolation inequality \eqref{E2}.
	%Let $z_0\in L^2(G)$, $t,s\in [\delta T,T]$ such that $s<t$ and $\tau\in (s,t)$. Using \eqref{eq:hypFCF1} and \eqref{eq:hypFCF2}, we obtain 
	%\[\|\mathbb{E}(z(T-t,\cdot;\xi; z_{T}))\|_{L^1(G_0)}\leq \frac{C_2}{C_1}\|\mathbb{E}( z(T-\tau,\cdot;\xi; z_{T}))\|_{L^1(G_0)}.\] 
	Let $\frac{s+t}{2}<\tau<t$, 
	\eqref{E2} yields:
	\begin{align*}
		&\|\mathbb{E}(I^{1-\alpha}_{t,T} z(T-t,\cdot;\xi; z_{T}))\|_{L^2(G)}\\
		&\leq  \left[C e^{C (t-s)^{-\sigma}} \|\mathbb{E}(z(T-\tau,\cdot;\xi; z_{T}))\|_{L^1(G_0)}\right]^{1-\kappa} \|\mathbb{E}(I^{1-\alpha}_{t,T} z(T-s,\cdot;\xi; z_{T}))\|_{L^2(G)}^\kappa.
	\end{align*}
	Then, by Young’s inequality, we deduce that:
	\begin{align*}
		&\|\mathbb{E}(I^{1-\alpha}_{t,T} z(T-t,\cdot;\xi; z_{T}))\|_{L^2(G)}\\
		&\leq  (1-\kappa)\varepsilon^{\frac{1}{1-\kappa}}C e^{C (t-s)^{-\sigma}} \|\mathbb{E}(z(T-\tau,\cdot;\xi; z_{T}))\|_{L^1(G_0)}+  \kappa \varepsilon^{-\frac{1}{\kappa}} \|\mathbb{E}(I^{1-\alpha}_{t,T} z(T-s,\cdot;\xi; z_{T}))\|_{L^2(G)} \\
		&\leq  \varepsilon^{\frac{1}{1-\kappa}}C e^{C (t-s)^{-\sigma}} \|\mathbb{E}(z(T-\tau,\cdot;\xi; z_{T}))\|_{L^1(G_0)}+   \varepsilon^{-\frac{1}{\kappa}} \|\mathbb{E}(I^{1-\alpha}_{t,T} z(T-s,\cdot;\xi; z_{T}))\|_{L^2(G)},
	\end{align*}
	for any $\varepsilon>0$ and $\frac{s+t}{2}<\tau<t$. Integrating the above inequality with respect to $\tau$ over $\left(\frac{s+t}{2},t\right)$, we obtain:
	\begin{align*}
		\|\mathbb{E}(I^{1-\alpha}_{t,T} z(T-t,\cdot;\xi; z_{T}))\|_{L^2(G)}
		&\leq  \varepsilon^{\frac{1}{1-\kappa}}\frac{C}{t-s} e^{C (t-s)^{-\sigma}} \int_{\frac{s+t}{2}}^{t}\|\mathbb{E}(z(T-\tau,\cdot;\xi; z_{T}))\|_{L^1(G_0)}\mathrm{d} \tau \\
		& \quad +   \varepsilon^{-\frac{1}{\kappa}} \|\mathbb{E}(I^{1-\alpha}_{t,T} z(T-s,\cdot;\xi; z_{T}))\|_{L^2(G)}
		\\&\leq  \varepsilon^{\frac{1}{1-\kappa}}\frac{C}{t-s} e^{C (t-s)^{-\sigma}} \int_{s}^{t}\|\mathbb{E}(z(T-\tau,\cdot;\xi; z_{T}))\|_{L^1(G_0)}\mathrm{d} \tau \\
		&\quad +  \varepsilon^{-\frac{1}{\kappa}} \|\mathbb{E}(I^{1-\alpha}_{t,T} z(T-s,\cdot;\xi; z_{T}))\|_{L^2(G)}.	
	\end{align*}
	By taking $\varepsilon = e^{(t-s)^{-\sigma}}$ and choosing $C$ sufficiently large, we obtain
	\begin{align*}
		&e^{-(C+\frac{1}{1-\kappa})(t-s)^{-\sigma}}\|\mathbb{E}(I^{1-\alpha}_{t,T} z(T-t,\cdot;\xi; z_{T}))\|_{L^2(G)} 
		-e^{-(C+\frac{1}{1-\kappa} + \frac{1}{\kappa})(t-s)^{-\sigma}}\|\mathbb{E}(I^{1-\alpha}_{t,T} z(T-s,\cdot;\xi; z_{T}))\|_{L^2(G)}\\
		&\leq  C  \int_{s}^{t}\|\mathbb{E}(z(T-\tau,\cdot;\xi; z_{T}))\|_{L^1(G_0)}\mathrm{d} \tau .
	\end{align*}
	We set 
	\[
	\mathfrak{C} = C + \frac{1}{1 - \kappa}, \qquad 
	q = 1 + \frac{1}{\kappa \mathfrak{C}} > 1.
	\]
	Then,
	\begin{align*}
		&e^{-\mathfrak{C}(t-s)^{-\sigma}}\|\mathbb{E}(I^{1-\alpha}_{t,T} z(T-t,\cdot;\xi; z_{T}))\|_{L^2(G)}-e^{-q \mathfrak{C} (t-s)^{-\sigma}}\|\mathbb{E}(I^{1-\alpha}_{t,T} z(T-s,\cdot;\xi; z_{T}))\|_{L^2(G)}\\
		&\leq  \mathfrak{C}  \int_{s}^{t}\|\mathbb{E}( z(T-\tau,\cdot;\xi; z_{T}))\|_{L^1(G_0)}\mathrm{d} \tau .
	\end{align*}
	For \( p = q^{-\frac{1}{\sigma}} < 1 \), we consider the decreasing sequence 
	\(\{t_m\}_{m \in \mathbb{N}}\) of elements in \([\delta T,\, T]\) defined by:
	\[
	t_m = \delta T + (1-\delta)p^{m}T, \qquad m \in \mathbb{N}.
	\]
	Then, by taking \( t = t_{m} \) and \( s = t_{m+1} \), we obtain:
	\begin{align*}
		&e^{-\mathfrak{C}(t_{m}-t_{m+1})^{-\sigma}}\|\mathbb{E}(I^{1-\alpha}_{t,T} z(T-t_{m},\cdot;\xi; z_{T}))\|_{L^2(G)}
		-e^{- \mathfrak{C} (t_{m+1}-t_{m+2})^{-\sigma}}\|\mathbb{E}(I^{1-\alpha}_{t,T} z(T-t_{m+1},\cdot;\xi; z_{T}))\|_{L^2(G)}\\
		&\leq  \mathfrak{C}  \int_{t_{m+1}}^{t_{m}}\|\mathbb{E}(z(T-\tau,\cdot;\xi; z_{T}))\|_{L^1(G_0)}\mathrm{d} \tau.
	\end{align*}
	By summing this latter inequality over \(\mathbb{N}\), we obtain:
	\begin{align*}
		e^{-\mathfrak{C}(T-t_{1})^{-\sigma}}\|\mathbb{E}(I^{1-\alpha}_{t,T} z(0,\cdot;\xi; z_{T}))\|_{L^2(G)} &\leq  \mathfrak{C}\sum_{m=0}^{\infty}  \int_{t_{m+1}}^{t_{m}}\|\mathbb{E}(z(T-\tau,\cdot;\xi; z_{T}))\|_{L^1(G_0)}\mathrm{d} \tau\\
		&=\mathfrak{C}\int_{\delta T}^{T}\|\mathbb{E}(z(T-\tau,\cdot;\xi; z_{T}))\|_{L^1(G_0)}\mathrm{d} \tau\\
		&=\mathfrak{C}\int_{0}^{(1-\delta)T}\|\mathbb{E}(z(\tau,\cdot;\xi; z_{T}))\|_{L^1(G_0)}\mathrm{d} \tau\\
		&\leq \mathfrak{C}\int_{0}^{T}\|\mathbb{E}(z(\tau,\cdot;\xi; z_{T}))\|_{L^1(G_0)}\mathrm{d} \tau,
	\end{align*}
	which establishes the observability inequality \eqref{est:obs}.
\end{proof}

The null controllability of the fractional biharmonic diffusion equation can be obtained by observing that its dynamics can be interpreted as the average of a fractional Schrödinger equation with diffusivity governed by Rademacher random variables, as indicated in the introduction.
\begin{cor}
	Let $\alpha\in (0,1)$. Then, the fractional biharmonic diffusion equation
	\begin{equation*} 
		\begin{cases}
			\partial^{\alpha}_{0,t} \tilde{y} + \Delta^2 \tilde{y}=\mathds{1}_{G_0} u,& \mbox{ on }Q_T,\\
			\Delta\tilde{y}=\tilde{y}=0, & \mbox{ on }\Sigma_T,\\
			\tilde{y}(0,\cdot)=y_0, & \mbox{ in } G,
		\end{cases}
	\end{equation*}
	is null controllable at any time $T>0$; that is, for any $y_0\in L^2(G)$, there exists $u\in L^{\infty}((0,T)\times G)$ such that $\tilde{y}(T,\cdot)=0$.
\end{cor}
This extends to the fractional case the null controllability result of \cite{guerrero2019} regarding the biharmonic heat equation. 

%\section{Some numerical results and experiments} \label{Sec5}
% This section aims to illustrate numerically the null averaged controllability of system \eqref{sys:frac Schr}.

%\section{Additional control problems}\label{sec:prob}
%\paragraph{}
%In this Section, we would like to describe the analog controllability 
%problems, and point out some open problems:

\section{Conclusion}\label{Sec6}
In this work, we have investigated the averaged controllability of the time-fractional Schrödinger equation, where the quantum diffusivity is modeled as a random variable following a general probability distribution. We first established that simultaneous null controllability can occur only for a countable set of realizations of the random diffusivity. We then showed that exact averaged controllability fails for all absolutely continuous random variables, independently of the control time. Moreover, for a broad class of random variables, we proved that the system is null-averaged controllable at any time from any subset of the spatial domain with positive measure, using an open-loop control that does not depend on the randomness.

Along the way, we rigorously established the duality between averaged controllability and averaged observability for the adjoint systems associated with time-fractional Schrödinger equations. To the best of our knowledge, this duality has not been rigorously proved in the existing literature, even for the special case of fractional equations with deterministic diffusion. Furthermore, we introduced a two-parameter fractional characteristic function, which generalizes the classical characteristic function and is of independent interest to probability and statistics communities. 

Our methodology is also applicable to abstract time-fractional Schrödinger equations possessing similar spectral properties, i.e., with self-adjoint operators having compact resolvent and satisfying spectral inequalities. This includes, for instance, other boundary conditions (e.g., of dynamic type \cite{mercado2023}) and degenerate equations (see, e.g., \cite{alahyane2025}).

Moreover, our theoretical results naturally lead to a numerical algorithm for computing average controls of minimal $L^2$-norm (HUM controls) associated with time-fractional Schrödinger equations; see \cite{barcena2025} for the integer case $\alpha=1$. However, the efficient treatment for time-fractional Schrödinger equations needs a detailed numerical analysis, which goes beyond the scope of our paper. This will eventually be investigated in a forthcoming paper.

\section*{Funding}
J.A.B.P was supported  by the grant IT1615-22 funded by the Basque Government.

\section*{Declaration of competing interest}
The authors declare that they have no known competing interests that could have appeared to influence the work reported in this paper.

\section*{Data availability}
No data was used for the research described in the article.

\end{document}